\documentclass{amsart}

\usepackage{amsmath,amssymb,amsfonts,amsthm}
\usepackage{hyperref}
\hypersetup{colorlinks=true,urlcolor=blue,linkcolor=blue,citecolor=red}
\usepackage[nameinlink,noabbrev]{cleveref}
\usepackage{tikz}
\usetikzlibrary{shapes,arrows,decorations.markings,positioning}
\usepackage{hypergraph}
\usepackage[colorinlistoftodos,bordercolor=orange,backgroundcolor=orange!20,linecolor=orange,textsize=scriptsize]{todonotes}
\usepackage{xspace}
\usepackage{enumitem}
\setlist[enumerate]{label=(\roman*),leftmargin=2.5em}
\setlist[itemize]{label={\raisebox{0.15ex}{\tiny \textbullet}},leftmargin=\parindent}

\theoremstyle{plain}
\newtheorem{theorem}{Theorem}[section]
\newtheorem{lemma}[theorem]{Lemma}

\newtheorem{corollary}[theorem]{Corollary}

\theoremstyle{definition}

\newtheorem{example}[theorem]{Example}

\newtheorem{problem}{Problem}
\Crefname{prb}{Problem}{Problems}

\renewcommand{\geq}{\geqslant}
\renewcommand{\leq}{\leqslant}
\newcommand{\Mbar}{\overline{\mathcal{M}}}
\newcommand{\MIN}{{\scshape Min}\xspace}
\newcommand{\MAX}{{\scshape Max}\xspace}
\newcommand{\R}{\mathbb{R}}
\newcommand{\N}{\mathbb{N}}
\newcommand{\Acal}{\mathcal{A}}
\newcommand{\Gcal}{\mathcal{G}}
\newcommand{\Hcal}{\mathcal{H}}
\newcommand{\Scal}{\mathcal{S}}
\newcommand{\unit}{e}
\newcommand{\tail}{\mathbf{t}}
\newcommand{\head}{\mathbf{h}}
\newcommand{\cpt}[1]{ {#1^c} }
\newcommand{\Cpt}[1]{{[n] \setminus #1}}
\newcommand{\Hilbert}[1]{\| \ifx\\#1\\ \cdot \else #1 \fi \|_\textup{H}}
\newcommand{\Rplus}{\R_{\scriptscriptstyle >0}}
\newcommand{\transpose}[1]{{#1}^\intercal}

\newcommand{\Tsr}{\mathcal{F}}

\DeclareMathOperator{\intr}{int}

\DeclareMathOperator{\reach}{reach}

\DeclareMathOperator{\supp}{supp}
\DeclareMathOperator{\argmin}{arg\,min}
\DeclareMathOperator{\argmax}{arg\,max}

\title[Nonlinear {P}erron-{F}robenius eigenvectors]
{A game theory approach to the existence and uniqueness of
nonlinear {P}erron-{F}robenius eigenvectors}

\date{\today}

\author[M.~Akian]{Marianne Akian}
\author[S.~Gaubert]{St\'ephane Gaubert}
\address{INRIA Saclay-Ile-de-France and CMAP, Ecole polytechnique,
Route de Saclay, 91128 Palaiseau Cedex, France}
\email{marianne.akian@inria.fr}
\email{stephane.gaubert@inria.fr}

\author[A.~Hochart]{Antoine Hochart}
\address{Universidad Adolfo Ib{\'a}{\~n}ez, Santiago, Chile}
\email{antoine.hochart@gmail.com}

\thanks{A.~Hochart is funded by FONDECYT grant 3180662.}


\begin{document}

\maketitle

\begin{abstract}
  We establish a generalized Perron-Frobenius theorem, based on a combinatorial criterion
  which entails the existence of an eigenvector for any nonlinear order-preserving and
  positively homogeneous map $f$ acting on the open orthant $\Rplus^n$.
  This criterion involves dominions, i.e., sets of states that can be made invariant by
  one player in a two-person game that only depends on the behavior of $f$
  ``at infinity''.
  In this way, we characterize the situation in which for all $\alpha, \beta > 0$,
  the ``slice space''
  $\Scal_\alpha^\beta := \{ x \in \Rplus^n \mid \alpha x \leq f(x) \leq \beta x \}$
  is bounded in Hilbert's projective metric, or, equivalently, for all uniform
  perturbations $g$ of $f$, all the orbits of $g$ are bounded
  in Hilbert's projective metric.
  This solves a problem raised by Gaubert and Gunawardena (Trans.\ AMS, 2004).
  We also show that the uniqueness of an eigenvector is characterized by
  a dominion condition, involving a different game depending now on the local behavior of
  $f$ near an eigenvector.
  We show that the dominion conditions can be verified by directed hypergraph methods.
  We finally illustrate these results by considering specific classes of nonlinear maps,
  including Shapley operators, generalized means and nonnegative tensors.
\end{abstract}

\section{Introduction}
\label{sec:intro}

\subsection{Motivations}

Nonlinear Perron-Frobenius theory \cite{LN12} deals with self-maps $f$ of a closed convex
cone $K$ in a Banach space, that are positively homogeneous of degree one and
preserve the partial order $\leq$ induced by $K$.
For brevity, we will refer to these as {\em monotone homogeneous} maps.
Such maps appear in several fields (sometimes up to immediate variations), including
population dynamics \cite{Nus89,thieme}, dynamical systems \cite{Nuss-Mallet,fathi08,Kra15},
optimal control and risk sensitive control \cite{FHH97,AG03,CCHH10,AB15},
zero-sum repeated games \cite{KM97,RS01,Sor04,Ren11,GV12}, or in the study of diffusions
on fractals \cite{sabot,metz}.
In all these fields, a basic question is to understand the asymptotic behavior of
the iterates $f^k$ as $k$ tends to infinity.

The dynamical behavior of $f$ is best understood when $f$ has an eigenvector
in the interior of $K$, i.e., when there exists a vector $u \in \intr K$ such that
$f(u) = \lambda u$ for some scalar $\lambda > 0$.
For instance, if the cone $K$ is polyhedral and finite-dimensional, this entails that
for all $x \in K$, $\lambda^{-k} f^k(x)$ converges to a periodic sequence whose length
can be bounded only in terms of the number of facets of $K$, see \cite[Thm.~2.2]{AGLN} and
Chapter~8 of \cite{LN12} for a broader perspective.
As pointed out by Nussbaum \cite[p.~4]{Nus88}, the existence of such an eigenvector is
``perhaps $\textup{[\,\dots]}\xspace$ the irreducible analytic difficulty''.
We point out that standard fixed point arguments only lead, under some compactness
assumptions, to the existence of an eigenvector in the {\em closed} cone $K$. 

This question is already non-obvious when $K$
is the standard nonnegative orthant of $\R^n$.
We denote by $\Rplus^n$ the interior of this cone, i.e., the set of all vectors in $\R^n$
whose entries are positive.
In this setting, the existence of a positive eigenvector $u$ with eigenvalue $\lambda$
implies that 
\[
  \lim_{k \to \infty} \left[ f^k(x) \right]_i^{1/k}  = \lambda
\]
for all $x$ in $\Rplus^n$ and all $i$ in $[n] := \{1,\dots,n\}$.
In particular, the limit is independent of the choice of $i$ and $x$
(see e.g., \cite[Sec.~2.2]{GG04}).
This limit represents a geometric ``escape rate'' which is of essential interest in applications
--- for instance, in population dynamics, it represents the rate of growth.

A useful tool in nonlinear Perron-Frobenius theory is {\em Hilbert's projective metric},
defined for $x, y \in \intr K$, and in particular when $x, y \in \Rplus^n$, by
\[
  d_\textup{H}(x,y) := \log \, \inf
  \left\{ \frac{\beta}{\alpha} \mid \alpha x \leq y \leq \beta x \right\} .
\]
The latter map satisfies all the axioms of a metric, except that $d_\textup{H}(x,y) = 0$
holds whenever $x$ and $y$ are proportional.
Hence, it yields a bona fide metric on the space of rays included in $\intr K$.
It is known that $f$ is nonexpansive with respect to $d_\textup{H}$, meaning that
$d_\textup{H} ( f(x), f(y) ) \leq d_\textup{H} (x, y)$.
In this way, the eigenproblem for $f$ can be thought of as a fixed point problem for
the nonexpansive map obtained by making $f$ act on the rays of $\intr K$.
A result of Nussbaum \cite[Thm.~4.1]{Nus88} implies that $f$ has an eigenvector if
and only if one orbit (or equivalently every orbit) of $f$ is bounded in Hilbert's
projective metric.

There are invariant sets that can be defined for all monotone homogeneous self-maps
of $\Rplus^n$, the so-called {\em slice spaces} (see \cite{GG04}):
\begin{equation*}
  \Scal_\alpha^\beta(f) := \left\{x \in \Rplus^n \mid \alpha x \leq f(x) \leq \beta x \right\} ,
  \quad \alpha , \beta > 0 .
   \label{e-slice}
\end{equation*}
Hence, a simple method to guarantee the existence of an eigenvector of $f$ consists in
finding a nonempty slice space $\Scal_\alpha^\beta(f)$ that is bounded in Hilbert's
projective metric.
Note that if $\alpha$ is small enough and $\beta$ large enough, then the slice space
$\Scal_\alpha^\beta(f)$ is not empty.
So, the difficulty in this approach resides in checking the boundedness.

Of particular interest is the situation in which the slice spaces $\Scal_\alpha^\beta(f)$
are bounded {\em for all} parameters $\alpha$ and $\beta$. 
Gaubert and Gunawardena showed that this is amenable to combinatorial methods, and
gave sufficient conditions, involving graphs or recession functions, implying that
all slice spaces are bounded in Hilbert's projective metric (see \cite[Thm.~4, Thm.~13]{GG04}),
leading to a generalized Perron-Frobenius theorem. 
They observed however that these conditions are suboptimal and formulated the following problem. 

\begin{problem}[{\cite[p.~4937]{GG04}}]
  \label[prb]{pb:disjoint-dominions-mult}
  Give a combinatorial characterization of the property that all slice spaces of
  a monotone homogeneous self-map of $\Rplus^n$ are bounded in Hilbert's projective metric.
\end{problem}

Requiring the boundedness of all slice spaces leads to an existence criterion for
eigenvectors which is robust to uniform perturbations of the map.
Indeed, it is shown in \cite[Thm.~5.1]{Hoc17},
extending \cite[Thm.~3.1]{AGH15}, that all the slice spaces of $f$ are bounded in Hilbert's
projective metric if and only if every monotone homogeneous map $g$ such that
$d_\textup{H}(f(x),g(x))$ is bounded independently of $x \in \Rplus^n$ has
a positive eigenvector.

A further motivation arises from a recent work of Lemmens, Lins and Nussbaum \cite{LLN18}, 
dealing with the boundedness of the eigenspace of $f$ in Hilbert's projective metric.
The latter is equivalent to the property that at least one slice space of $f$ is nonempty
and bounded in this metric. 
Whereas this property is hard to check, we shall see that requiring the boundedness of
{\em all} slice spaces is computationally more tractable.

\subsection{Description of the main results}

In the present paper, we provide effective criteria inspired by game theory
to check the existence and uniqueness of nonlinear eigenvectors,
and solve \Cref{pb:disjoint-dominions-mult}.

To this end, we associate to any monotone homogeneous map $f: \Rplus^n \to \Rplus^n$
a two-player stochastic game $\Gamma_\infty(f)$ defined as follows.
The state space is $[n]$ and when in state $i \in [n]$, the action space of
the first player, called \MIN (resp., the second player, called \MAX), is composed of
the sets $J \subset [n]$ such that
\begin{align*}
  \label{e-rel}
  \lim_{\alpha \to +\infty} f_i \big( \exp( \alpha \unit_\cpt{J} ) \big) < +\infty
  \qquad \left( \text{resp.,} \enspace
  \lim_{\alpha \to -\infty} f_i \big( \exp( \alpha \unit_\cpt{J}) \big) > 0 \right) .
\end{align*}
Here, $f_i$ denotes the $i$th coordinate function of $f$; $\cpt{J} := \Cpt{J}$ is
the complement of $J$ in $[n]$; for any set $K \subset [n]$, $\unit_K$ is the vector
of $\R^n$ with entries equal to $1$ on $K$ and $0$ elsewhere; and $\exp: \R^n \to \Rplus^n$
is the map which applies the exponential componentwise.
The game is played repeatedly, starting from a given initial position.
At each stage, the players choose simultaneously a set $I$ and $J$, respectively,
and the next state is drawn uniformly in $I \cap J$.
We will see that this intersection is never empty, so that the game is well posed.

We call {\em dominion} a nonempty subset of states $\Delta \subset [n]$ subject to the control
of one player, in the sense that from any initial position in $\Delta$, that player can force
the state to remain almost surely in $\Delta$ at each stage, whatever actions the other player chooses.

Our first main result is the following.

\begin{theorem}[Generalized Perron-Frobenius theorem]
  \label{cor:solvability-eigenproblem}
  Let $f: \Rplus^n \to \Rplus^n$ be a monotone homogeneous map.
  If the two players do not have disjoint dominions in the game 
  $\Gamma_\infty(f)$, then $f$ has an eigenvector.
\end{theorem}

This theorem provides a considerable extension of the nonlinear Perron-Frobenius theorems
in \cite{GG04}, where it is shown that the same conclusion holds if a certain digraph
associated to $f$ is strongly connected (Theorem~2) and more generally, if $f$ is
indecomposable (Theorem~6).
The condition of \Cref{cor:solvability-eigenproblem} is less demanding, as illustrated
in \Cref{sec:stochgames}.

The absence of disjoint dominions in the game $\Gamma_\infty(f)$ is a combinatorial condition.
Our next result reveals the importance of this condition by showing that it is
{\em equivalent} to the boundedness of all the slice spaces of $f$, therefore
solving \Cref{pb:disjoint-dominions-mult}.
This condition is also equivalent to the existence of eigenvectors for all suitable
perturbations of the map $f$, and it can also be characterized in terms of
the dynamical behavior of such perturbed maps.
More precisely, we say that a monotone homogeneous map $g: \Rplus^n \to \Rplus^n$ 
is a {\em uniform perturbation} of $f$, with respect to Hilbert's projective metric, if
there exists a constant $\gamma \geq 1$ and, for all $x \in \Rplus^n$,
positive reals $\alpha_x, \beta_x$ satisfying $\frac{\beta_x}{\alpha_x} \leq \gamma$
and such that $\alpha_x f(x) \leq g(x) \leq \beta_x f(x)$.

\begin{theorem}
  \label{thm:existence-multiplicative}
  Let $f: \Rplus^n \to \Rplus^n$ be a monotone homogeneous map.
  The following assertions are equivalent:
  \begin{enumerate}
    \item the two players do not have disjoint dominions in the game $\Gamma_\infty(f)$;
      \label{it:dominion-condition}
    \item all the slice spaces of $f$ are bounded in Hilbert's projective metric;\label{it:main-thm-ii}
    \item for all diagonal matrices $D$ having positive diagonal entries, 
      the map $Df$ has an eigenvector;
      \label{it:main-thm-iii}
    \item for all diagonal matrices $D$ having positive diagonal entries, the limit
      \[
        \lim_{k \to \infty} \left[ (Df)^k(x) \right]_i^{1/k}
      \]
      exists for all $x \in \Rplus^n$ and all $i \in [n]$, and is independent of
      the choice of $x$ and $i$;
      \label{it:main-thm-iv}
    \item every monotone homogeneous map $g$ that is a uniform perturbation of $f$ has
      an eigenvector;
      \label{it:main-thm-v}
    \item for every monotone homogeneous map $g$ that is a uniform perturbation of $f$,
      the limit
      \[
        \lim_{k \to \infty} \left[ g^k(x) \right]_i^{1/k}  
      \]
      exists for all $x \in \Rplus^n$ and all $i \in [n]$, and is independent of
      the choice of $x$ and $i$.
      \label{it:main-thm-vi}
  \end{enumerate}
\end{theorem}

Here, the main contribution consists of the implications
\ref{it:dominion-condition}~$\Rightarrow$~\ref{it:main-thm-ii} and
\ref{it:main-thm-iv}~$\Rightarrow$~\ref{it:dominion-condition}, leading to
a full characterization.
The implication \ref{it:main-thm-ii}~$\Rightarrow$~\ref{it:main-thm-iii} follows
from \cite[Thm.~3]{GG04}.
The implications \ref{it:main-thm-iii}~$\Rightarrow$~\ref{it:main-thm-v}~$\Rightarrow$~\ref{it:main-thm-ii}
were established in~\cite[Thm.~5.1]{Hoc17}.
The other implications are either easy or straightforward. 

\medskip
We now turn our attention to the question of the {\em uniqueness} of the eigenvector.
Surprisingly, this is still controlled by dominions, but in a different game $\Gamma_u(f)$
defined in \Cref{sec:uniqueness}, depending only on the local behavior of $f$
at point $u$.

\begin{theorem}
  \label{thm:dominion-uniqueness-mult}
  Suppose that $f: \Rplus^n \to \Rplus^n$ is a monotone homogeneous map and that
  it has a positive eigenvector $u$.
  Then, $u$ is the unique positive eigenvector of $f$, up to a scalar factor,
  if and only if the players do not have disjoint dominions in the game $\Gamma_u(f)$.
\end{theorem}

This should be compared with a result established by Nussbaum and the first two authors \cite{AGN16}
in an infinite-dimensional setting.
This result relies on the notion of semidifferential.
Recall that $f$ is said to be semidifferentiable at point $u$ if there is a continuous and
positively homogeneous map $f'_u$ such that $f(u+h) = f(u) + f'_u(h) + o(\|h\|)$.
(This is similar to the definition of Fr{\'e}chet differentiability, however,
the semidifferential map $f'_u$ is not required to be linear which makes
the notion adapted to situations in which $f$ is nonsmooth.)
When specialized to the finite-dimensional case, Corollary~7.7 of \cite{AGN16} shows that
if $f$ is semidifferentiable at point $u$, then the uniqueness of the eigenvector of $f'_u$,
up to a scalar factor, entails the uniqueness of the eigenvector $u$, up to a scalar factor.
The latter condition is only sufficient.
The ``local game'' $\Gamma_u(f)$ may be thought of as a combinatorial refinement of
the semidifferential map $f'_u$, taking into account variations $f(u+h)$ of higher order
than $O(h)$, leading now to a necessary and sufficient condition.

\medskip
The paper is organized as follows.
We prove \Cref{thm:existence-multiplicative} in \Cref{sec:existence},
whereas \Cref{thm:dominion-uniqueness-mult} is proved
in \Cref{sec:uniqueness}.
To this end, we shall use ``logarithmic glasses'', considering the conjugate map
$T = \log \circ f\circ \exp$ instead of $f$.
We will see that $T$ can be thought of as the Shapley operator of a zero-sum game,
explaining the somehow unexpected occurrence of a game-theoretical condition
in the study of a nonlinear eigenproblem.

In \Cref{sec:algorithmic-aspects}, we show that the dominion condition can be
checked by solving reachability problems in directed {\em hypergraphs},
in which hyperedges connect {\em sets} of nodes.
This is in contrast with earlier works in Perron-Frobenius theory \cite{Nus89,GG04,CCHH10},
in which {\em graph} conditions were used to ensure the existence of a positive eigenvector.
Hypergraph conditions are tighter and considerably less demanding than graph conditions.
We note that the present hypergraph approach was initiated in \cite{AGH15}
in a more special setting (zero-sum games with {\em bounded} payments).

We apply in \Cref{applications} these results to concrete classes of maps.
These include nonlinear maps involving generalized means arising from mathematical biology
and matrix scaling problems, studied by Nussbaum in \cite{Nus89},
as well as nonnegative tensors.

Some of our results were announced in the conference paper \cite{AGH15-CDC}.

\section{Games, dominions and existence of eigenvectors}
\label{sec:existence}

\subsection{The additive setting}

The map $\log: \Rplus^n \to \R^n$, which applies the logarithm componentwise, is
a bijection between the standard positive cone $\Rplus^n$ and the space $\R^n$.
We denote by $\exp: \R^n \to \Rplus^n$ its inverse map.
These two maps are monotone (in the sense of preserving the order). 
Hence, a monotone homogeneous map $f$ on $\Rplus^n$ is conjugated to the map
$T = \log \circ f \circ \exp: \R^n \to \R^n$, which is monotone and {\em additively homogeneous}.
The latter property means that for all scalars $\alpha \in \R$ and all $x \in \R^n$,
we have $T(x + \alpha \unit) = T(x) + \alpha \unit$, where $\unit$ denotes
the unit vector of $\R^n$.
Then, through ``log glasses'', the nonlinear eigenproblem $f(u) = \lambda u$ is equivalent
to finding a pair $(\lambda,u) \in \R \times \R^n$ solution of the following equation,
known as the {\em ergodic equation}:
\begin{equation}
  \label{eq:ergodic-equation}
  T(u) = \lambda \unit + u .
\end{equation}
By analogy with the multiplicative case, we shall call $\lambda$ the eigenvalue of $T$,
which is unique, and $u$ an eigenvector.

Although the multiplicative and additive frameworks are equivalent, some tools are easier
to manipulate in one setting than in the other.
In particular, Hilbert's projective metric is replaced in $\R^n$ by {\em Hilbert's seminorm},
a.k.a.\ {\em Hopf's oscillation}:
\[
  \Hilbert{x} := \max_{1 \leq i \leq n} x_i - \min_{1 \leq i \leq n} x_i ,
  \quad x \in \R^n .
\]
Then, all the slice spaces $\Scal_\alpha^\beta(f)$ are bounded in Hilbert's projective
metric if and only if all the {\em additive slice spaces} of $T$, defined by
\[
  \Acal_\alpha^\beta(T) := \left\{x \in \R^n \mid \alpha \unit + x \leq T(x) \leq \beta \unit + x \right\} ,
  \quad \alpha, \beta \in \R ,
\]
are bounded in Hilbert's seminorm. 
In the sequel, we shall work with monotone additively homogeneous self-maps of $\R^n$,
leaving to the reader the immediate translation to the multiplicative framework. 

\subsection{Dominion condition}

We now fix a monotone additively homogeneous map $T: \R^n \to \R^n$ and define
a two-player stochastic game, denoted by $\Gamma_\infty(T)$, which coincides with the game
associated to the conjugate map $\exp \circ T \circ \log$ and defined in the introduction.
Precisely, the state space is $[n]$ and at each stage, if the current state is $i$, player \MIN
chooses a subset of states $I$ such that
\[
  \lim_{\alpha \to +\infty} T_i(\alpha \unit_\cpt{I}) < +\infty .
\]
Simultaneously, player \MAX chooses a subset of states $J$ such that
\[
  \lim_{\alpha \to -\infty} T_i(\alpha \unit_\cpt{J}) > -\infty .
\]
Then, the state at the next stage is chosen in $I \cap J$ with uniform probability.
We are only interested in the dynamics of the state, and therefore we do not need to
define a payoff function.

Let us make some observations.
First, since $T$ is monotone, the condition for $I$ to be in the action space of \MIN
in state $i$ is equivalent to the nondecreasing sequence
$\{ T_i(k \unit_\cpt{I}) \}_{k \in \N}$ being bounded.
Likewise, $J$ is in the action space of \MAX in state $i$ if and only if the nonincreasing
sequence $\{ T_i(-k \unit_\cpt{J}) \}_{k \in \N}$ is bounded.
The monotonicity of $T$ also implies that if a player can choose a set $I$
in a given state, then he can choose any set $K \supset I$.
Furthermore, in each state, the action spaces of the two players are never empty since
they contain the set of all states $[n]$.
Finally, observe that they do not contain the empty set.

For the game $\Gamma_\infty(T)$ to be properly defined, we need that for all sets $I$ and $J$
chosen in any state $i$ by \MIN and \MAX respectively, the set of possible next states,
$I \cap J$, be nonempty.
We next show that this is the case.
Suppose, by way of contradiction, that $I \cap J = \emptyset$.
Then we have $\unit - \unit_\cpt{J} = \unit_J \leq \unit_\cpt{I}$.
By monotonicity and additive homogeneity of $T$, we deduce that
$\alpha + T_i(-\alpha \unit_\cpt{J}) \leq T_i(\alpha \unit_\cpt{I})$
for all scalars $\alpha \geq 0$.
The latter inequality entails that $\lim_{\alpha \to +\infty} T_i(\alpha \unit_\cpt{I})$
and $\lim_{\alpha \to -\infty} T_i(\alpha \unit_\cpt{J})$ cannot be both finite,
a contradiction.

\medskip
A {\em dominion} of some player is a nonempty subset of states $\Delta \subset [n]$ such that,
from any initial position $i \in \Delta$, this player can force the state to remain
almost surely in $\Delta$ at each stage.
Equivalently, $\Delta$ is a dominion of some player in $\Gamma_\infty(T)$ if he can choose
in each state $i \in \Delta$ a set $I \subset [n] $ such that $I \cap J \subset \Delta$
for any admissible choice $J$ of the adversarial player.  
Note that, since $J = [n]$ is always an admissible choice for this other player,
$\Delta$ is a dominion of the former if and only if in each state $i \in \Delta$
he is allowed to choose the set $\Delta$.

We note that the notion of {\em contra-ergodic partition}, equivalent
to the notion of disjoint dominions for a special class of games, was introduced by
Gurvich and Lebedev in \cite{GL89} and further used by Boros, Elbassioni, Gurvich and Makino
in~\cite{BEGM10} for algorithmic purposes.
We also mention that Jurdszi\'nski, Paterson and Zwick introduced in \cite{JPZ08}
the notion of dominion to develop an algorithm to solve a class of combinatorial two-player
games called parity games.
In their setting, all the states in a dominion of a player are also required to be winning
for this player.
This condition is not relevant here, since the game has no payoff.

We can now state our main result, in the additive setting, from which
\Cref{thm:existence-multiplicative} and then
\Cref{cor:solvability-eigenproblem} are readily derived.
In this setting, a uniform perturbation (with respect to Hilbert's seminorm) of
a self-map $T$ of $\R^n$ is any map $G: \R^n \to \R^n$ for which $\Hilbert{G(x) - T(x)}$
is uniformly bounded, i.e., for which there exists a constant $\gamma > 0$ and,
for all $x \in \R^n$, real numbers $\alpha_x, \beta_x$ satisfying
$\beta_x - \alpha_x \leq \gamma$ and such that
$\alpha_x \unit + T(x) \leq G(x) \leq \beta_x \unit + T(x)$.

\begin{theorem}
  \label{thm:existence-additive}
  Let $T: \R^n \to \R^n$ be monotone and additively homogeneous.
  The following assertions are equivalent:
  \begin{enumerate}
    \item the two players do not have disjoint dominions in the game $\Gamma_\infty(T)$;
      \label{it:existence-i}
    \item all the additive slice spaces of $T$ are bounded in Hilbert's seminorm;
      \label{it:existence-ii}
    \item for all $g \in \R^n$, there exists $(\lambda,u) \in \R \times \R^n$
      such that $g+T(u) = \lambda \unit + u$;
      \label{it:existence-iii}
    \item for all $g \in \R^n$, the limit
      \[
        \lim_{k \to \infty} \left[ \frac{(g+T)^k(x)}{k} \right]_i
      \]
      exists for all $x \in \R^n$ and $i \in [n]$, and is independent of
      the choice of $x$ and $i$;
      \label{it:existence-iv}
    \item for all monotone additively homogeneous maps $G$ that are uniform perturbations
      of $T$, there exists $(\lambda,u) \in \R \times \R^n$ such that
      $G(u) = \lambda \unit + u$;
      \label{it:existence-v}
    \item for all monotone additively homogeneous maps $G$ that are uniform perturbations
      of $T$, the limit
      \[
        \lim_{k \to \infty} \left[ \frac{G^k(x)}{k} \right]_i
      \]
      exists for all $x \in \R^n$ and $i \in [n]$, and is independent of
      the choice of $x$ and $i$.
      \label{it:existence-vi}
  \end{enumerate}
\end{theorem}

Before giving the proof, let us illustrate the result.

\begin{example}
  \label{ex:running-example}
  Let $T: \R^3 \to \R^3$ be the monotone additively homogeneous map defined by
  \begin{equation}
    \label{eq:running-ex}
    T(x) =
    \begin{pmatrix}
      \frac{1}{2} (x_1 + x_2) \, \wedge \, \left( -1 + \frac{1}{2} (x_1 + x_3) \right) \\
      \left( 1 + \frac{1}{2} (x_1 + x_3) \right) \, \wedge \, \left( \frac{1}{2} (x_1 + x_2) \,
      \vee \, (-3 + x_3 ) \right) \\
      ( 1 + \frac{1}{2} (x_1 + x_3) ) \, \vee \, ( 1 + x_3 )
    \end{pmatrix} ,
  \end{equation}
  where $\alpha \wedge \beta = \min \{\alpha, \beta\}$ for any $\alpha, \beta \in \R$.

  The game $\Gamma_\infty(T)$ has $3$ states.
  In state 1 we have, for instance, $T_1( \alpha \unit_{\{3\}} ) = -1 + \alpha / 2$ if
  $\alpha \leq 2$ and $T_1( \alpha \unit_{\{3\}} ) = 0$ otherwise.
  Hence $\lim_{\alpha \to +\infty} T_1( \alpha \unit_{\{3\}} ) < +\infty$ and
  $\lim_{\alpha \to -\infty} T_1( \alpha \unit_{\{3\}} ) = - \infty$.
  So the set $\{1,2\}$ is an admissible action for player \MIN
  but not for player \MAX.
  More generally, the actions of player \MIN are
  \begin{itemize}
    \item in state 1: $\{1,2\}$, $\{1,3\}$, $\{1,2,3\}$;
    \item in state 2: $\{1,3\}$, $\{1,2,3\}$;
    \item in state 3: $\{1,3\}$, $\{1,2,3\}$.
  \end{itemize}
  Hence the dominions of \MIN are $\{1,3\}$ and $\{1,2,3\}$.
  As for player \MAX, his actions are
  \begin{itemize}
    \item in state 1: $\{1,2,3\}$;
    \item in state 2: $\{1,3\}$, $\{1,2,3\}$;
    \item in state 3: $\{3\}$, $\{1,3\}$, $\{2,3\}$, $\{1,2,3\}$.
  \end{itemize}
  Hence the dominions of \MAX are $\{3\}$ and $\{1,2,3\}$.

  Consequently, the players do not have disjoint dominions in $\Gamma_\infty(T)$.
  We deduce from \Cref{thm:existence-additive} that all the additive slice spaces of $T$
  are bounded in Hilbert's seminorm.
  In particular, the ergodic equation~\labelcref{eq:ergodic-equation} is solvable
  (check that $T(u) = u$ with $u = \transpose{(0,0,2)}$).
\end{example}

\subsection{Proof of \Cref{thm:existence-additive}}

We already know that \ref{it:existence-ii}~$\Leftrightarrow$~\ref{it:existence-iii}~$\Leftrightarrow$~\ref{it:existence-v}
(see \cite[Thm.~9]{GG04} and \cite[Thm.~5.1]{Hoc17}).
The implication \ref{it:existence-v}~$\Rightarrow$~\ref{it:existence-vi}
(as well as \ref{it:existence-iii}~$\Rightarrow$~\ref{it:existence-iv}) is
an easy observation, which is made in particular in \cite[Sec.~2.2]{GG04}.
The implication \ref{it:existence-vi}~$\Rightarrow$~\ref{it:existence-iv} is trivial.
Hence it suffices to show that \ref{it:existence-i}~$\Rightarrow$~\ref{it:existence-ii} and
\ref{it:existence-iv}~$\Rightarrow$~\ref{it:existence-i}.
We next prove these implications by showing their contrapositive.

\subsubsection*{Proof of \ref{it:existence-i}~$\Rightarrow$~\ref{it:existence-ii}}

Suppose that there is an additive slice space $\Acal_\alpha^\beta(T)$ which is unbounded
in Hilbert's seminorm.
So there exists a sequence $\{ u^k \}_{k \in \N}$ in $\Acal_\alpha^\beta(T)$ such that
$\lim_{k \to \infty} \Hilbert{u^k} = \infty$.
For all integers $k$ we may assume, up to the addition by a vector proportional to $\unit$,
that $\min_{\ell \in [n]} u^k_\ell = 0$ , and we let
$\tau^k = \max_{\ell \in [n]} u^k_\ell = \Hilbert{u^k}$.

Let $I$ be a subset of $[n]$ with maximal cardinality among all the subsets
$L \subset [n]$ for which there exists a subsequence $ \{ u^{n_k} \}_{k \in \N}$ such that
$\{ u^{n_k}_\ell \}_{k \in \N}$ is bounded for all $\ell \in L$.
Such a set $I$ is nonempty since $\min_{\ell \in [n]} u^k_\ell = 0$ for all $k \in \N$.
Furthermore $I \neq [n]$ since $\tau^k = \Hilbert{u^k}$ tends to infinity.

Let $\phi: \N \to \N$ be a strictly increasing function such that $\{ u^{\phi(k)} \}_{k \in \N}$
is a subsequence as described above for $I$.
Then $\lim_{k \to \infty} u^{\phi(k)}_\ell = +\infty$ for all $\ell \notin I$.
Otherwise, there would be some $j \notin I$ such that a subsequence of
$\{ u^{\phi(k)}_j \}_k$ is bounded.
Hence, there would be a subsequence $\{ u^{n_k} \}_k$ such that
$\{ u^{n_k}_\ell \}_k$ is bounded for all $\ell \in I \cup \{j\}$, a contradiction
with the maximality of the cardinality of $I$.

The set $I$ being fixed, we can show, following the same lines, that there exists
a set $J \subset [n]$ for which there is a strictly increasing function $\psi: \N \to \N$
such that $\{ \tau^{\phi(\psi(k))} - u^{\phi( \psi(k) )}_j \}_k$ is bounded for all
$j \in J$ and $\tau^{\phi( \psi(k) )} - u^{\phi( \psi(k) )}_\ell$ tends to infinity
for all $\ell \notin J$.
Since $\lim_{k \to \infty} \tau^{\phi(k)} = \infty$, such a set $J$ is nonempty,
and since $\{ u^{\phi(k)}_i \}_k$ is bounded for all $i \in I$, we have $I \cap J = \emptyset$.

For the sake of simplicity, assume that $\psi \circ \phi$ is the identity map, so that
\begin{itemize}
  \item for all $i \in I$, the sequence $\{ u^k_i \}_k$ is bounded;
  \item for all $j \in J$, the sequence $\{ \tau^k - u^k_j \}_k$ is bounded;
  \item for all $\ell \notin I \cup J$, the sequences $\{ u^k_\ell \}_k$ and
    $\{ \tau^k - u^k_\ell \}_k$ tend to $+\infty$.
\end{itemize}

\medskip
For all $k \in \N$, let us denote $\rho^k = \min_{\ell \notin I} u^k_\ell$ and
$\sigma^k = \max_{\ell \notin J} u^k_\ell$.
By construction, we have
$\lim_{k \to \infty} \rho^k = \lim_{k \to \infty} \tau^k - \sigma^k = +\infty$.
We also have, for all integers $k$,
\[
  \rho^k \unit_\cpt{I} \leq u^k \quad \text{and} \quad 
  (\tau^k - \sigma^k) \unit_\cpt{J} \leq \tau^k \unit - u^k .
\]

Let $M > 0$ be a joint upper bound for the sequences $\{ u^k_i \}_k$ with $i \in I$,
and $\{ \tau^k - u^k_j \}_k$ with $j \in J$.
Then we deduce from the previous observations that, for all indices $i \in I$,
\[
  T_i(\rho^k \unit_\cpt{I}) \leq T_i(u^k) \leq \beta + u_i^k \leq \beta + M .
\]
This proves that in the game $\Gamma_\infty(T)$, \MIN can choose the set $I$
in each state $i \in I$.
Hence, $I$ is a dominion of \MIN.
Likewise, for all indices $j \in J$, we have
\[
  T_j\left( (\sigma^k - \tau^k) \unit_\cpt{J} \right) \geq
  T_j(u^k - \tau^k \unit) \geq \alpha + u^k_j - \tau^k \geq \alpha - M ,
\]
which proves that $J$ is a dominion of \MAX, and thus that the players do have disjoint
dominions in the game $\Gamma_\infty(T)$.
\qed

\subsubsection*{Proof of \ref{it:existence-iv}~$\Rightarrow$~\ref{it:existence-i}}

Assume that there exist disjoint dominions, $I$ for player \MIN and $J$ for player \MAX,
in the game $\Gamma_\infty(T)$.
Then, by definition of the game, there exist constants $\alpha, \beta \in \R$ such that
\begin{alignat}{2}
  \label{eq:dominion-MIN}
  & T_i(k \unit_\cpt{I}) \leq \beta , && \quad \forall i \in I , \quad \forall k \geq 0 , \\
  \label{eq:dominion-MAX}
  & k + \alpha \leq T_j(k \unit_J) , && \quad \forall j \in J , \quad \forall k \geq 0 .
\end{alignat}
Moreover, we can choose $\alpha$ and $\beta$ so that $\alpha \leq T_\ell(0) \leq \beta$
for all $\ell \in [n]$.

Let $s \in \N$ and define the map $S: x \mapsto s \unit_J + T(x)$.
We next show by induction that, for all integers $k$,
\begin{equation}
  \label{eq:bound-mean-value}
  k (s \unit_J + \alpha \unit) \leq S^k(0) \leq k (s \unit_\cpt{I} + \beta \unit) .
\end{equation}
For $k=0$, this is trivial.
Now assume that \labelcref{eq:bound-mean-value} holds for some $k \in \N$.
Then we have
\begin{equation*}
  \begin{split}
    S^{k+1}(0) = s \unit_J + T( S^k(0) )
    & \leq s \unit_\cpt{I} + T(k s \unit_\cpt{I}) + k \beta \unit , \\
    & \leq s \unit_\cpt{I} + (k s \unit_\cpt{I} + \beta \unit) + k \beta \unit , \\
    & = (k+1) ( s \unit_\cpt{I} + \beta \unit) .
  \end{split}
\end{equation*}
The first inequality follows from the fact that $\unit_J \leq \unit_\cpt{I}$ and from
the second inequality in \labelcref{eq:bound-mean-value}.
The second inequality follows from \labelcref{eq:dominion-MIN} for the coordinates $i \in I$
and from the fact that $T_\ell (k s \unit_\cpt{I}) \leq T_\ell (k s \unit) = k s + T_\ell(0)$
for the coordinates $\ell \notin I$.
Likewise, we have
\begin{equation*}
  \begin{split}
    S^{k+1}(0) = s \unit_J + T( S^k(0) )
    & \geq s \unit_J + T(k s \unit_J) + k \alpha \unit , \\
    & \geq s \unit_J + (k s \unit_J + \alpha \unit) + k \alpha \unit , \\
    & = (k+1) ( s \unit_J + \alpha \unit) ,
  \end{split}
\end{equation*}
where the first inequality follows from the induction hypothesis \labelcref{eq:bound-mean-value} and
the second inequality, from \labelcref{eq:dominion-MAX}.
Thus \labelcref{eq:bound-mean-value} holds for $k+1$ which concludes the proof by induction.

Since the inequalities \labelcref{eq:bound-mean-value} hold for all integers $k$, we deduce that
\[
  s \unit_J + \alpha \unit \leq \liminf_{k \to \infty} \frac{S^k(0)}{k} \leq
  \limsup_{k \to \infty} \frac{S^k(0)}{k} \leq s \unit_\cpt{I} + \beta \unit .
\]
So, if $\chi$ denotes an arbitrary accumulation point of
the sequence $\{ k^{-1} S^k(0)\}_{k \geq 1}$, we have
\begin{alignat*}{2}
  & \alpha \leq \chi_i \leq \beta , && \quad \forall i \in I , \\
  & s + \alpha \leq \chi_j \leq s + \beta , && \quad \forall j \in J .
\end{alignat*}
Choosing $s > \beta - \alpha$, we conclude that $\chi$ cannot be a constant vector,
and so, condition \labelcref{it:existence-iv} cannot be satisfied.
\qed

\subsection{The special case of convex maps}
\label{sec:convexity-slice-spaces}

In this subsection, we consider a monotone additively homogeneous operator
$T: \R^n \to \R^n$ which is convex, meaning that every coordinate function of $T$ is convex.
In this particular case, we show that \Cref{thm:existence-additive} simplifies.
More precisely, the two-player game $\Gamma_\infty(T)$ can be replaced by
a one-player game in which the actions of \MIN are discarded while the actions of \MAX
are essentially the same as in $\Gamma_\infty(T)$.

Before introducing this one-player game, let us give the following definition, which applies
to any real map $g: \R^n \to \R$:
we call {\em support} of $g$, and we denote it by $\supp(g)$, the set of indices $i \in [n]$
such that $g$ depends effectively on $x_i$, in the sense that there is no map
$h: \R^{n-1} \to \R$ such that $g(x) = h(x_1,\dots,x_{i-1},x_{i+1},\dots,x_n)$.
The next lemma provides a characterization of the support of a convex map $g$
that is monotone and additively homogeneous, i.e., that satisfies
$x \leq y \implies g(x) \leq g(y)$ for all $x, y \in \R^n$ and $g(x + \alpha \unit) = 
g(x) + \alpha$ for all $x \in \R^n$ and $\alpha \in \R$.

\begin{lemma}[{\cite[Prop.~2]{GG04}}]
  Let $g: \R^n \to \R$ be convex, monotone and additively homogeneous.
  An index $i \in [n]$ is contained in $\supp(g)$, the support of $g$, if and only if
  $\lim_{\alpha \to +\infty} g(\alpha \unit_{\{i\}}) = +\infty$.
  \label{lem:support}
\end{lemma}

We then deduce a characterization of the set of actions of player \MIN in $\Gamma_\infty(T)$.

\begin{corollary}
  \label{cor:action-MIN-convex}
  Let $T: \R^n \to \R^n$ be convex, monotone and additively homogeneous.
  Then, in any state $i \in [n]$, a set $I$ is an action of player \MIN in $\Gamma_\infty(T)$
  if and only if $I \supset \supp(T_i)$.
\end{corollary}

\begin{proof}
  Let $i \in [n]$ and $I \subset [n]$.
  If $I \supset \supp(T_i)$, then $\cpt{I} \subset \cpt{\supp(T_i)}$, and so
  $T_i(\alpha \unit_\cpt{I})$ is independent of $\alpha$, which implies that
  $I$ can be chosen by \MIN in the game $\Gamma_\infty(T)$ when in state $i$.
  Conversely, assume that $I$ can be chosen by \MIN in $\Gamma_\infty(T)$ when in state $i$.
  By monotonicity of $T$ we have, for all $j \in \cpt{I}$ and all positive scalars $\alpha$,
  $T_i(\alpha \unit_{\{j\}}) \leq T_i(\alpha \unit_\cpt{I})$.
  Since the right-hand side of the latter inequality is bounded by a constant
  independent of $\alpha$, we deduce from \Cref{lem:support} that $j \notin \supp(T_i)$.
  This yields that $\supp(T_i) \subset I$.
\end{proof}

\medskip
We now define a one-player game $\breve{\Gamma}_\infty(T)$ as follows.
The state space is $[n]$ and when in state $i$ the player, called \MAX, chooses
a subset of states $J$ such that
\[
  J \subset \supp(T_i) \quad \text{and} \quad
  \lim_{\alpha \to -\infty} T_i(\alpha \unit_\cpt{J}) > -\infty .
\]
Once this set is selected, the next state is chosen in $J$ with uniform probability.
For the same reason as with $\Gamma_\infty(T)$, we do not need to define a payoff function.

The game $\breve{\Gamma}_\infty(T)$ is well defined since in any state $i \in [n]$, the set of
actions of player \MAX is nonempty (it contains the action $J = \supp(T_i)$)
and does not contain the empty set.
Furthermore, we have the following straightforward connection between the actions of \MAX
in $\Gamma_\infty(T)$ and in $\breve{\Gamma}_\infty(T)$.

\begin{lemma}
  Let $T: \R^n \to \R^n$ be convex, monotone and additively homogeneous.
  Any action of player \MAX in $\breve{\Gamma}_\infty(T)$ is also an action of \MAX
  in $\Gamma_\infty(T)$.
  Conversely, for every state $i \in [n]$, if a subset of states $J$ is an action
  of \MAX in $\Gamma_\infty(T)$, then $J \cap \supp(T_i)$ is an action of \MAX
  in $\breve{\Gamma}_\infty(T)$.
  \label{lem:actions-MAX-convex}
  \qed
\end{lemma}

The notion of dominion in $\breve{\Gamma}_\infty(T)$ is defined as in $\Gamma_\infty(T)$.
In particular a set $\Delta$ is a dominion of \MAX if in any state $i \in \Delta$ he can
choose a set $J \subset \Delta$.
\Cref{lem:actions-MAX-convex} yields that a subset of states is
a dominion of \MAX in the two-player game $\Gamma_\infty(T)$ if and only if it is a dominion in
the one-player game $\breve{\Gamma}_\infty(T)$.

We shall also need the following notion.
We say that a nonempty set $\Theta \subset [n]$ is {\em invariant} in $\breve{\Gamma}_\infty(T)$
if for every initial position in $\Theta$, the state remains almost surely in $\Theta$
at each stage, whatever action is chosen by player \MAX.
This is equivalent to the condition that, in every state in $\Theta$, the action space of \MAX
contains only sets $J \subset \Theta$.
By definition of $\breve{\Gamma}_\infty(T)$ and using \Cref{cor:action-MIN-convex},
we get the following equivalences.

\begin{lemma}
  \label{lem:invariant-dom}
  A set $\Theta \subset [n]$ is invariant in $\breve{\Gamma}_\infty(T)$ if and only if
  $\supp(T_i) \subset \Theta$ for all $i \in \Theta$, or if and only if $\Theta$ is
  a dominion of player \MIN in $\Gamma_\infty(T)$.
  \qed
\end{lemma}

The above observations on dominions and invariant sets in $\breve{\Gamma}_\infty(T)$ lead to
the following adaptation of \Cref{thm:existence-additive} to
the case of convex maps.

\begin{corollary}
  \label{cor:bounded-slice-spaces-convex}
  A convex monotone additively homogeneous map $T: \R^n \to \R^n$ 
  has a slice space $\Acal_\alpha^\beta(T)$ that is unbounded in Hilbert's seminorm
  if and only if there exists in the one-player game $\breve{\Gamma}_\infty(T)$ an invariant set
  disjoint from a dominion of player \MAX.
\end{corollary}

\section{Games, dominions and uniqueness of eigenvectors}
\label{sec:uniqueness}

In this section, we give, for any monotone additively homogeneous map, a game-theoretical
characterization of the uniqueness of the eigenvector, up to an additive constant
(i.e., up to the addition by a multiple of the unit vector).
Remarkably enough, the criterion turns out to be identical, up to the definition of the game,
to the dominion condition given in \Cref{sec:existence}.

Let $T$ be a monotone additively homogeneous self-map of $\R^n$,
and $u$ be a point in $\R^n$.
We introduce an abstract ``local'' two-player stochastic game $\Gamma_u(T)$ as follows.
The state space is $[n]$ and when in state $i$, player \MIN chooses a set $I \subset [n]$
such that
\begin{equation}
  \label{eq:uniqueness-action-min}
  \exists \varepsilon > 0 , \quad \forall \alpha \in [0,\varepsilon] , \quad
  T_i(u + \alpha \unit_\cpt{I}) = T_i(u) .
%  \exists \varepsilon > 0 , \quad T_i(u + \varepsilon \unit_\cpt{I}) = T_i(u) .
\end{equation}
Dually, player \MAX chooses in state $i$ a set $J \subset [n]$ such that
\begin{equation}
  \label{eq:uniqueness-action-max}
  \exists \varepsilon > 0 , \quad \forall \alpha \in [0,\varepsilon] , \quad
  T_i(u - \alpha \unit_\cpt{J}) = T_i(u) .
%  \exists \varepsilon > 0 , \quad T_i(u - \varepsilon \unit_\cpt{J}) = T_i(u) .
\end{equation}
Then the next state is drawn uniformly in $I \cap J$.
Similarly to the game $\Gamma_\infty(T)$, we do not need to define a payoff function since
we are only interested in the state dynamics.

Before stating the subsequent uniqueness result, let us show that the game
$\Gamma_u(T)$ is well defined.
First, both players can always choose the set $[n]$, so that their action spaces are not empty.
Moreover, for every admissible choice $(I,J)$ in any state $i$, we have
$I \cap J \neq \emptyset$ (in particular $I$ and $J$ are not empty).
Indeed, if $I$ and $J$ are two disjoint sets, then we have
$\unit_J = \unit - \unit_\cpt{J} \leq \unit_\cpt{I}$, which implies that
for every $\varepsilon \geq 0$ 
\[
  T_i(u) \leq \varepsilon + T_i(u - \varepsilon \unit_\cpt{J}) \leq
  T_i(u + \varepsilon \unit_\cpt{I}) \leq \varepsilon + T_i(u) .
\]
It follows that the conditions \labelcref{eq:uniqueness-action-min} and
\labelcref{eq:uniqueness-action-max} cannot be both satisfied.
Further note that, by monotonicity of $T$, if \MIN (resp., \MAX) can choose some set $I$,
then he can choose any other set $K \supset I$.

Dominions are defined as in \Cref{sec:existence} and since
the ``abstract'' games $\Gamma_\infty(T)$ and $\Gamma_u(T)$ are identical,
up to the definition of the action spaces, the following characterization holds:
a subset of states $\Delta$ is a dominion of one player in $\Gamma_u(T)$ if and only if,
for each state $i \in \Delta$, the set $\Delta$ is an admissible action for that player.
We can now state the game-theoretical criterion for uniqueness of eigenvectors.

\begin{theorem}
  \label{thm:uniqueness-additive}
  Let $T: \R^n \to \R^n$ be a monotone additively homogeneous map.
  An eigenvector $u$ of $T$ is unique, up to an additive constant,
  if and only if the players do not have disjoint dominions in the game $\Gamma_u(T)$.
\end{theorem}

Before proving this theorem, we provide an illustration.
\begin{example}
  Consider the map $T: \R^3 \to \R^3$ introduced in \Cref{ex:running-example}:
  \begin{equation*}
    T(x) =
    \begin{pmatrix}
      \frac{1}{2} (x_1 + x_2) \, \wedge \, \left( -1 + \frac{1}{2} (x_1 + x_3) \right) \\
      \left( 1 + \frac{1}{2} (x_1 + x_3) \right) \, \wedge \, \left( \frac{1}{2} (x_1 + x_2) \,
      \vee \, (-3 + x_3 ) \right) \\
      ( 1 + \frac{1}{2} (x_1 + x_3) ) \, \vee \, ( 1 + x_3 )
    \end{pmatrix} .
  \end{equation*}
  We know that the eigenvalue of $T$ is $0$ and that $u = \transpose{(0,0,2)}$ is
  an eigenvector.
  Let us check whether this is the unique eigenvector, up to an additive constant.

  The game $\Gamma_u(T)$ has $3$ states.
  In state 1 we have, for instance, $T_1 (u + \alpha \unit_{\{3\}}) = 0 = T_1(u)$
  for all $\alpha \geq 0$.
  Hence the set $\{1,2\}$ is an admissible action for player \MIN.
  On the other hand, $T_1 (u - \alpha \unit_{\{3\}}) = - \alpha / 2 < T_1(u)$
  for all $\alpha > 0$.
  So $\{1,2\}$ is not an admissible action for player \MAX.
  More generally, the actions of player \MIN are
  \begin{itemize}
    \item in state 1: $\{1,2\}$, $\{1,3\}$, $\{1,2,3\}$;
    \item in state 2: $\{1,2\}$, $\{1,2,3\}$;
    \item in state 3: $\{1,3\}$, $\{1,2,3\}$.
  \end{itemize}
  Hence the dominions of \MIN are $\{1,2\}$, $\{1,3\}$ and $\{1,2,3\}$.
  As for player \MAX, his actions are
  \begin{itemize}
    \item in state 1: $\{1,2,3\}$;
    \item in state 2: $\{1,2\}$, $\{1,2,3\}$;
    \item in state 3: $\{3\}$, $\{1,3\}$, $\{2,3\}$, $\{1,2,3\}$.
  \end{itemize}
  Hence the dominions of \MAX are $\{3\}$ and $\{1,2,3\}$.

  Consequently, $(\{1,2\},\{3\})$ is a pair of disjoint dominions in $\Gamma_u(T)$,
  and we deduce from \Cref{thm:uniqueness-additive} that $u$ is not
  the unique eigenvector of $T$, up to an additive constant.
  Indeed, it can be checked that $v = \transpose{(0,0,3)}$ is another eigenvector.
\end{example}

% Remark.
% In the latter example we have, for $I = \{1,2\}$,
% $T_2(u + \alpha \unit_\cpt{I}) = 0$ if $0 \leq \alpha \leq 1$,
% $-1 + \alpha$ if $1 \leq \alpha \leq 6$ and $2 + \alpha/2$ if $6 \leq \alpha$.

\begin{proof}[Proof of \Cref{thm:uniqueness-additive}]
  Beforehand, let us note that if we replace $T$ by the map
  $x \mapsto T(u+x) - \lambda \unit - u$ (where $\lambda \in \R$ is the eigenvalue of $T$),
  we might as well assume that $\lambda = 0$ and $u = 0$.

  Suppose that $v$ is a nonconstant eigenvector of $T$ (i.e., such that $v = v - u$ is not
  proportional to $\unit$).
  Up to the addition of a constant vector, we can assume that $\min_{\ell \in [n]} v_\ell = 0$.
  Let $I = \argmin v := \{ i \in [n] \mid v_i = 0 \}$.
  Then there exists a scalar $\varepsilon > 0$ such that $\varepsilon \unit_\cpt{I} \leq v$
  and we have $0 = T(0) \leq T(\varepsilon \unit_\cpt{I}) \leq T(v) = v$.
  This implies that $T_i(\varepsilon \unit_\cpt{I}) = 0$ for all $i \in I$.
  Thus, $I$ is a dominion of player \MIN in $\Gamma_u(T)$.

  Assume now (up to the addition of a constant) that $\max_{\ell \in [n]} v_\ell = 0$ and let
  $J = \argmax v := \{ i \in [n] \mid v_i = 0 \}$.
  Then there exists a scalar $\varepsilon' > 0$ such that $-\varepsilon' \unit_\cpt{J} \geq v$.
  This entails $0 = T(0) \geq T(-\varepsilon' \unit_\cpt{J}) \geq T(v) = v$, which implies
  that $T_j(-\varepsilon' \unit_\cpt{J}) = 0$ for all $j \in J$.
  Hence, $J$ is a dominion of player \MAX in $\Gamma_u(T)$, and since
  $I \cap J \neq \emptyset$, this proves that the dominion condition holds in $\Gamma_u(T)$.

  \medskip
  Conversely, assume that $(I,J)$ is a pair of disjoint dominions of \MIN and \MAX,
  respectively, in $\Gamma_u(T)$.
  Then, recalling that $T(u) = u = 0$, there exists $\varepsilon > 0$ such that
  \begin{alignat}{2}
    \label{eq:FP-argmin}
    & T_i(\varepsilon \unit_\cpt{I}) = 0 , && \quad \forall i \in I , \\
    \label{eq:FP-argmax}
    & T_j(-\varepsilon \unit_\cpt{J}) = 0 , && \quad \forall j \in J .
  \end{alignat}
  If $v$ is a vector in $\R^n$ which satisfies
  \begin{equation}
    \begin{aligned}
      & v_i = 0 , && \forall i \in I , \\
      & 0 \leq v_\ell \leq \varepsilon , && \forall \ell \notin I \cup J , \\
      & v_j = \varepsilon , && \forall j \in J ,
    \end{aligned}
    \label{eq:nonconstant-vector}
  \end{equation}
  then we have
  $0 \leq \varepsilon \unit_J = \varepsilon \unit - \varepsilon \unit_\cpt{J} \leq v \leq
  \varepsilon \unit_\cpt{I} \leq \varepsilon \unit$.
  Using~\labelcref{eq:FP-argmin} and~\labelcref{eq:FP-argmax} this yields
  \begin{equation}
    \label{eq:FP-argmin-argmax}
    \begin{aligned}
      & T_i(v) = 0 = v_i , \quad && \forall i \in I , \\
      & T_j(v) = \varepsilon = v_j , \quad && \forall j \in J .
    \end{aligned}
  \end{equation}

  If $I \cup J = [n]$, then we readily obtain a nonconstant eigenvector of $T$.
  Otherwise, let $L = \Cpt{I \cup J}$.
  We introduce the map $G: \R^L \to \R^n$ defined by
  \[
    G_\ell(x) =
    \begin{cases}
      0 & \quad \text{if} \quad \ell \in I , \\
      x_\ell & \quad \text{if} \quad \ell \in L , \\
      \varepsilon & \quad \text{if} \quad \ell \in J ,
    \end{cases}
  \]
  and the map $F: \R^L \to \R^L$ defined by $F_\ell(x) = T_\ell \left( G(x) \right)$
  for all $\ell \in L$.
  If $x \in [0,\varepsilon]^L$, then we have $G(x) \in [0,\varepsilon]^n$, which implies
  \[
    0 = T_\ell(0) \leq T_\ell \left( G(x) \right) \leq \varepsilon + T_\ell(0) = \varepsilon ,
    \quad \forall \ell \in L .
  \]
  This shows that $F$ maps the compact convex set $[0,\varepsilon]^L$ to itself.
  Hence, according to Brouwer's fixed point theorem, there exists $x^* \in [0,\varepsilon]^L$
  such that $F(x^*) = x^*$.

  Let $v = G(x^*)$ so that $T_\ell(v) = v_\ell$ for all $\ell \in L$.
  Since $v$ satisfies~\labelcref{eq:nonconstant-vector}, we also have~\labelcref{eq:FP-argmin-argmax}.
  Therefore, $v$ is a nonconstant eigenvector of $T$.
\end{proof}

\section{Algorithmic aspects}
\label{sec:algorithmic-aspects}

In this section, we give a graph-theoretical construction that allows us to check
the dominion condition in $\Gamma_\infty(T)$ and $\Gamma_u(T)$.
In the same way the latter games have a common structure, their combinatorial counterparts
that we hereafter introduce share identical properties.
To avoid the repetition of similar arguments, we then carry out a detailed analysis of
combinatorial and complexity aspects related to the first problem (existence of an eigenvector),
and present more briefly the results related to the second (uniqueness of the eigenvector). 

\subsection{Preliminaries on hypergraphs}

A {\em directed hypergraph} $\Hcal$ is a pair $(N,A)$ where $N$ is a set of {\em nodes} and
$A$ is a set of (directed) {\em hyperarcs}.
A hyperarc $a$ is an ordered pair $(\tail(a),\head(a))$ of disjoint nonempty subsets of nodes;
$\tail(a)$ is the {\em tail} of $a$ and $\head(a)$ is its {\em head}.
For brevity, we shall write $\tail$ and $\head$ instead of $\tail(a)$ and $\head(a)$,
respectively.
When $\tail$ and $\head$ are both of cardinality one, the hyperarc is said to be an arc,
and when every hyperarc is an arc, the directed hypergraph becomes a directed graph.
In the following, the term hypergraph will always refer to a directed hypergraph.

We will also need the notion of reachability in $\Hcal=(N,A)$.
A {\em hyperpath} of length $p$ from a set of nodes $I \subset N$ to a node $j \in N$
is a sequence of $p$ hyperarcs $(\tail_1,\head_1),\dots,(\tail_p,\head_p)$, such that
$\tail_i \subset \bigcup_{k=0}^{i-1} \head_k$ for all $i=1,\dots,p+1$
with the convention $\head_0 = I$ and $\tail_{p+1} = \{j\}$.
We say that a node $j \in N$ is {\em reachable} from a subset $I$ of $N$ if there exists
a hyperpath from $I$ to $j$.
Alternatively, the relation of reachability can be defined in a recursive way:
$j$ is reachable from $I$ if either $j \in I$ or there exists a hyperarc $(\tail,\head)$
such that $j \in \head$ and every node of $\tail$ is reachable from the set $I$.
A subset $J$ of $N$ is said to be {\em reachable} from a subset $I$ of $N$ if
every node of $J$ is reachable from $I$.
We denote by $\reach(I,\Hcal)$ the set of reachable nodes from $I$ in $\Hcal$.
A subset $I$ of $N$ is {\em invariant} in the hypergraph $\Hcal$ if it contains all the nodes
that are reachable from itself, i.e., $\reach(I,\Hcal) \subset I$, hence $\reach(I,\Hcal) = I$
since the other inclusion always holds.
One readily checks that, for $J \subset N$, $\reach(J,\Hcal)$ is the smallest invariant set
in $\Hcal$ containing $J$.

For background on hypergraphs, we refer the reader to \cite{All14} and the references
therein, and in particular to \cite{GLNP93} for reachability problems.

\clearpage

\subsection{Existence of eigenvectors}
\label{sec:existence-hypergraphs}

\subsubsection{Hypergraphs and dominions}
\label{sec:hypergraphs-dominions-slice-spaces}

Given a monotone additively homogeneous map $T: \R^n \to \R^n$, we introduce a pair of
hypergraphs, denoted by $(\Hcal_\infty^+(T),\Hcal_\infty^-(T))$, and defined as follows:
\begin{itemize}
  \item the set of nodes of $\Hcal_\infty^+(T)$ and $\Hcal_\infty^-(T)$ is $[n]$;
  \item the hyperarcs of $\Hcal_\infty^+(T)$ are the pairs $(J,\{i\})$ such that $i \notin J$ and
    \[
      \lim_{\alpha \to +\infty} T_i(\alpha \unit_J) = +\infty ;
    \]
  \item the hyperarcs of $\Hcal_\infty^-(T)$ are the pairs $(J,\{i\})$ such that $i \notin J$ and
    \[
      \lim_{\alpha \to -\infty} T_i(\alpha \unit_J) = -\infty .
    \]
\end{itemize}

Equivalently, we can reformulate the definition of a hyperarc $(J,\{i\})$ in $\Hcal^+_\infty(T)$
(resp., $\Hcal^-_\infty(T)$) by asking that $i \notin J$ and $\cpt{J}$ is not an action of
player \MIN (resp., \MAX) in $\Gamma_\infty(T)$ when in state $i$.
A straightforward consequence of this definition is that a subset of nodes $J \subset [n]$
is invariant in $\Hcal_\infty^+(T)$ (resp., $\Hcal_\infty^-(T)$) if and only if,
for every $i \in \cpt{J}$, $\cpt{J}$ is an action of player \MIN (resp., player \MAX)
in $\Gamma_\infty(T)$.
This leads to the following characterization.

\begin{lemma}
  \label{lem:invariant-subsets}
  A set of nodes $J \subsetneq [n]$ is invariant in the hypergraph $\Hcal_\infty^+(T)$
  (resp., $\Hcal_\infty^-$(T)) if and only if its complement $\cpt{J}$ is a dominion of
  player \MIN (resp., \MAX) in the game $\Gamma_\infty(T)$.
  \qed
\end{lemma}

Therefore, \Cref{thm:existence-additive} can be reformulated in terms of hypergraph reachability.

\begin{theorem}
  \label{thm:hypergraphs-slice-spaces}
  A monotone additively homogeneous map $T: \R^n \to \R^n$ has a slice space
  $\Acal_\alpha^\beta(T)$ that is unbounded in Hilbert's seminorm if and only if
  there exist a pair $(I,J)$ of nonempty disjoint subsets of $[n]$ such that
  $\reach(\cpt{I},\Hcal_\infty^+(T)) = \cpt{I}$ and $\reach(\cpt{J},\Hcal_\infty^-(T)) = \cpt{J}$.
  \qed
\end{theorem}

\begin{example}
  Consider the map $T: \R^3 \to \R^3$ introduced by \labelcref{eq:running-ex}
  in \Cref{ex:running-example} and let us construct the hypergraphs $\Hcal_\infty^\pm(T)$.
  Their set of nodes is $\{1,2,3\}$.
  In $\Hcal_\infty^+(T)$, for instance, there is no arc from $\{2\}$ to $\{1\}$ since
  $T_1(\alpha \unit_{\{2\}}) = -1$ for all $\alpha \geq 0$, hence
  $\lim_{\alpha \to +\infty} T_1(\alpha \unit_{\{2\}}) < +\infty$.
  Likewise, since $T_1(\alpha \unit_{\{3\}}) = 0$ for all $\alpha \geq 2$,
  there is no arc from $\{3\}$ to $\{1\}$.
  However, $\lim_{\alpha \to +\infty} T_1(\alpha \unit_{\{2,3\}}) =
  \lim_{\alpha \to +\infty} (-1 + \alpha / 2) = +\infty$,
  so there is a hyperarc from $\{2,3\}$ to $\{1\}$.

  Dually, $T_1(\alpha \unit_{\{2\}}) = \alpha / 2$ for all $\alpha \leq -2$ and
  $T_1(\alpha \unit_{\{3\}}) = -1 + \alpha / 2$ for all $\alpha \leq 2$.
  So $\lim_{\alpha \to -\infty} T_1(\alpha \unit_{\{2\}}) =
  \lim_{\alpha \to -\infty} T_1(\alpha \unit_{\{3\}}) = -\infty$, which implies that
  there is an arc from $\{2\}$ to $\{1\}$ and one from $\{3\}$ to $\{1\}$ in $\Hcal_\infty^-(T)$.

  \Cref{fig:hypergraph-running-ex} shows a concise representation of these hypergraphs,
  where only the (hyper)arcs with minimal tail (with respect to the inclusion partial order)
  are represented.

  \begin{figure}[htbp]
    \begin{center}
      \begin{minipage}[b]{0.33\linewidth}
        \centering
        \begin{tikzpicture}
          [scale=1,on grid,auto,bend angle=40,
            state/.style={circle,draw,inner sep=0pt,minimum size=0.4cm},
            head/.style={<-,>=angle 60},
            tail/.style={->,>=angle 60},
          hyperedge/.style={>=angle 60}]
          \node[state] (s1) at (0,0) {$1$};
          \node[state] (s2) at (-0.75,1.3) {$2$}
          edge [head,bend left] (s1);
          \node[state] (s3) at (-1.5,0) {$3$}
          edge [head,bend right] (s1)
          edge [tail,bend left] (s2);
          \hyperedgewithangles[0.35][$(hyper@tail)!0.65!(hyper@head)$]{s2/-90,s3/30}{-30}{s1/150};
        \node[state,color=black!0] (s5) at (-2.25,+1.5) [label=below:{$\Hcal_\infty^+(T)$}] {};
        \end{tikzpicture}
      \end{minipage}
      \hspace{3em}
      \begin{minipage}[b]{0.33\linewidth}
        \centering
        \begin{tikzpicture}
          [scale=1,on grid,bend angle=40,auto,
            state/.style={circle,draw,inner sep=0pt,minimum size=0.4cm},
            head/.style={<-,>=angle 60},
            tail/.style={->,>=angle 60}]
          \node[state] (s1) at (0,0) {$1$};
          \node[state] (s2) at (-0.75,1.3) {$2$}
          edge [tail,bend left] (s1)
          edge [head, bend right] (s1);
          \node[state] (s3) at (-1.5,0) {$3$}
          edge [tail,bend right] (s1)
          edge [tail,bend left] (s2);
        \node[state,color=black!0] (s5) at (-2.25,+1.5) [label=below:{$\Hcal_\infty^-(T)$}] {};
        \end{tikzpicture}
      \end{minipage}
    \end{center}
    \caption{The hypergraphs $\Hcal_\infty^\pm(T)$ associated with $T$ defined by \labelcref{eq:running-ex}.}
    \label{fig:hypergraph-running-ex}
  \end{figure}

  The only nontrivial invariant set of nodes of $\Hcal_\infty^+(T)$ (resp., $\Hcal_\infty^-(T)$)
  is $\{2\}$ (resp., $\{1,2\}$).   
  Since their complements are not disjoint, we deduce from \Cref{thm:hypergraphs-slice-spaces}
  that all the slice spaces of $T$ are bounded in Hilbert's seminorm.
\end{example}

\subsubsection{The special case of convex maps}
\label{sec:convexity-graph-slice-spaces}

We now show that for a convex monotone additively homogeneous map $T: \R^n \to \R^n$,
the reachability condition of \Cref{thm:hypergraphs-slice-spaces} simplifies.
To that purpose, we associate to $T$ the directed graph $\Gcal_\infty(T)$ with set of vertices
$[n]$ and an edge from $i$ to $j$ if 
\[
  \lim_{\alpha \to +\infty} T_i(\alpha \unit_{\{j\}}) = +\infty .
\]
Since $T$ is convex, this is equivalent to $j \in \supp(T_i)$ (see \Cref{lem:support}).

A {\em final class} of $\Gcal_\infty(T)$ is a nonempty set of nodes $C$ such that every
two nodes of $C$ are connected by a directed path, and every path starting from a node
in $C$ remains in it.
Recalling that a set $\Theta \subset [n]$ is invariant in the one-player game
$\breve{\Gamma}_\infty(T)$ if and only if $\supp(T_i) \subset \Theta$ for all $i \in \Theta$
(see \Cref{lem:invariant-dom}), a direct consequence of the definitions is that a final class
of $\Gcal_\infty(T)$ is an invariant set in $\breve{\Gamma}_\infty(T)$, and conversely that
any invariant set in $\breve{\Gamma}_\infty(T)$ contains a final class of $\Gcal_\infty(T)$.
Then the translation of \Cref{cor:bounded-slice-spaces-convex} in terms of graph
leads to the following.

\begin{corollary}
  A convex monotone additively homogeneous map $T:\R^n\to\R^n$ has all its slice spaces
  $\Acal_\alpha^\beta(T)$ bounded in Hilbert's seminorn if and only if the digraph
  $\Gcal_\infty(T)$ has a unique final class $C$ and $\reach(C,\Hcal_\infty^-(T)) = [n]$.
  \label{cor:hypergraphs-convex-slice-spaces}
\end{corollary}

\begin{proof}
  According to \Cref{cor:bounded-slice-spaces-convex}, we need to show that in the one-player
  game $\breve{\Gamma}_\infty(T)$, player \MAX has a dominion disjoint from an invariant set if
  and only if the directed graph $\Gcal_\infty(T)$ has more than one final class, or a unique final
  class which does not have access to the whole set of nodes in $\Hcal_\infty^-(T)$.

  First suppose that $\Gcal_\infty(T)$ has two distinct final classes.
  Then these sets are both invariant in $\breve{\Gamma}_\infty(T)$.
  Since any invariant set is also a dominion  of \MAX, it follows that the dominion/invariant
  set condition holds.

  Next, assume that $\Gcal_\infty(T)$ has a unique final class $C$ and that
  $\reach(C,\Hcal_\infty^-(T))$ is not $[n]$.
  Let $\Delta = \Cpt{\reach(C,\Hcal_\infty^-(T))}$.
  The latter set and $C$ are nonempty and disjoint.
  Furthermore, $C$ is invariant and $\Delta$ is by construction a dominion of \MAX in 
  $\Gamma_\infty(T)$, hence in $\breve{\Gamma}_\infty(T)$.
  So the dominion/invariant set condition holds.

  Now, assume that $\Theta$ is invariant in $\breve{\Gamma}_\infty$(T) and that $\Delta$ is
  a dominion of player \MAX such that $\Theta \cap \Delta = \emptyset$.
  If $\Gcal_\infty(T)$ has a unique final class, let us denote it by $C$.
  Then we necessarily have $C \subset \Theta$, since any invariant set in $\breve{\Gamma}_\infty$
  contains a final class of $\Gcal_\infty(T)$.
  Hence, $C$ and $\Delta$ are disjoint, that is, $C \subset \cpt{\Delta}$, which yields
  $\reach(C,\Hcal_\infty^-(T)) \subset \reach(\cpt{\Delta},\Hcal_\infty^-(T)) =
  \cpt{\Delta} \neq [n]$.
\end{proof}

\subsubsection{Complexity aspects}
\label{sec:complexity-slice-spaces}

Given a monotone additively homogeneous self-map $T$ on $\R^n$, the basic issue under
consideration is to check whether the eigenproblem~\labelcref{eq:ergodic-equation}, or
its multiplicative counterpart, is solvable.
\Cref{thm:hypergraphs-slice-spaces} (or \Cref{cor:hypergraphs-convex-slice-spaces}
in the convex case) provides a combinatorial condition for this property to hold.
This condition can be effectively checked as soon as the action spaces in $\Gamma_\infty(T)$,
which arise in the definition of the hyperarcs of $\Hcal_\infty^\pm(T)$, can be identified.
This is possible when the limits $\lim_{\alpha \to \pm \infty} T_i(\alpha \unit_J)$
can be computed, which happens in general situations (see \Cref{sec:class-M}).

To set aside the latter problem, it is convenient to introduce the map $\Omega_\infty(T)$,
called {\em oracle}, which takes as input $(J,i,\pm)$ and returns a yes/no answer,
the answer being ``yes'' if and only if $T_i(\alpha \unit_J)$ tends to $\pm\infty$
when $\alpha$ goes to $\pm \infty$. 
A Turing machine with oracle $\Omega_\infty(T)$ is a Turing machine which can send a query
to $\Omega_\infty(T)$ and use the answer.
A call to the oracle is counted as one computational step of the Turing machine.
We refer the reader to \cite{AB09} for a detailed presentation of oracle Turing machines.

We will need the following result, which gives a bound for the time required to compute
the set of reachable nodes in $\Hcal_\infty^\pm(T)$ from any set.

\begin{lemma}
  For any set $J \subset [n]$, $\reach(J,\Hcal_\infty^\pm(T))$ can be computed in $O(n^2)$
  steps by a Turing machine with oracle $\Omega_\infty(T)$.
  \label{lem:reachability-computation}
\end{lemma}

\begin{proof}
  Set $J_1 = J$.
  If $T_i(\alpha \unit_{J_1})$ remains bounded as $\alpha \to \pm \infty$
  for all $i \notin J_1$, then $J_1$ is invariant in $\Hcal_\infty^\pm(T)$, meaning that
  $J_1 = \reach(I,\Hcal_\infty^\pm(T))$.
  Otherwise, define $J_2$ as the union of $J_1$ and all the nodes $i \notin J_1$
  for which $T_i(\alpha \unit_{J_1})$ tends to $\pm \infty$ as $\alpha \to \pm \infty$.
  Repeating the same steps, we arrive at a set $J_k$ for some integer $k \leq n$,
  which is invariant in $\Hcal_\infty^\pm(T)$ and contains $J$.
  Hence, we must have $\reach(J,\Hcal_\infty^\pm(T)) \subset J_k$, since
  $\reach(J,\Hcal_\infty^\pm(T))$ is the smallest invariant set containing $J$.
  The other inclusion being trivial, we get $J_k = \reach(J,\Hcal_\infty^\pm(T))$.

  Now observe that each step $\ell$ requires $|\cpt{(J_\ell)}|$ calls to the oracle
  $\Omega_\infty(T)$ (where $|X|$ denotes the cardinality of any set $X$) and that the number of
  elementary operations is linear with respect to the number of calls.
  Hence the result.
  In particular, the number of calls is bounded by $\sum_{\ell=1}^n \ell \leq n^2$.
\end{proof}

It readily follows from the definition of dominions that the problem of deciding
whether a set $J$ is a dominion can be solved in $O(|J|)$ steps by a Turing machine
with oracle $\Omega_\infty(T)$.
Furthermore, it is easily seen that the condition of \Cref{thm:hypergraphs-slice-spaces} boils
down to check that, for every $I \subset [n]$, either
$\reach(\cpt{I},\Hcal_\infty^+(T)) \neq \cpt{I}$ or $\reach(I,\Hcal_\infty^-(T)) = [n]$.
Then we get the following.

\begin{theorem}
  Let $T$ be a monotone additively homogeneous self-map of $\R^n$.
  The problem of deciding whether all slice spaces are bounded in Hilbert's seminorm
  can be solved in $O(2^n n^2)$ steps by a Turing machine with oracle $\Omega_\infty(T)$.
  \label{thm:boundedness-decision}
\end{theorem}

This should be compared with the generalized Perron-Frobenius theorem of \cite{GG04}.
It is shown there that all the additive sub-eigenspaces
(i.e., the sets $\mathcal{A}^\beta := \{ x \in \R^n \mid T(x) \leq \beta e + x\}$ where
$\beta \in \R$) are bounded in Hilbert's seminorm if and only if a certain digraph constructed
by an aggregation procedure is strongly connected.
This leads to a simpler test, requiring only a polynomial number of calls to the oracle.
However, the condition checked in this way is only a sufficient one for the boundedness of
all the slice spaces.

Furthermore, the exponential bound in the above theorem cannot be reduced to a polynomial
bound unless P~$=$~NP. 
Indeed, a restricted version of this problem, concerning deterministic Shapley operators
with finite action spaces, reduces to the nonexistence of a nontrivial fixed point
of a monotone Boolean function, a problem shown to be coNP-hard by Yang and Zhao \cite{YZ04}.
See also \cite{AGH15} for more information on complexity issues.

\medskip
When $T$ is convex, the condition in \Cref{cor:hypergraphs-convex-slice-spaces}
requires the computation of the final classes of the directed graph $\Gcal_\infty(T)$.
This graph has $n$ nodes, and so, its strongly connected components can be found
in $O(n^2)$ steps, using Tarjan's algorithm.
This leads to the following bound.

\begin{corollary}
  \label{cor:boundedness-decision-convex}
  Let $T: \R^n \to \R^n$ be a convex monotone additively homogeneous map.
  The problem of deciding whether all slice spaces are bounded in Hilbert's seminorm
  can be solved in $O(n^2)$ steps by a Turing machine with oracle $\Omega_\infty(T)$.
\end{corollary}

\subsection{Uniqueness of eigenvectors}
\label{sec:hypergraphs-uniqueness}

\subsubsection{Hypergarphs and dominions}

The construction of the hypergraphs in \Cref{sec:hypergraphs-dominions-slice-spaces}
(which is based only on the action spaces of the game $\Gamma_\infty(T)$) can be readily
transposed to $\Gamma_u(T)$.
Thus, up to the game with which the hypergraphs are associated, the results in
\Cref{sec:hypergraphs-dominions-slice-spaces} provides a graph-theoretical characterization
of the dominion condition in $\Gamma_u(T)$.
We next briefly present these results.

Following the definition of $\Hcal_\infty^+(T)$ and $\Hcal_\infty^-(T)$, we associate to
any monotone additively homogeneous self-map $T$ of $\R^n$ and any point $u \in \R^n$,
two hypergraphs, $\Hcal_u^+(T)$ and $\Hcal_u^-(T))$ respectively, with set of nodes $[n]$ and
a hyperarc $(J,\{i\})$ in $\Hcal_u^+(T)$ (resp., $\Hcal_u^-(T)$) if $i \notin J$ and $\cpt{J}$
is not an action of player \MIN (resp., player \MAX) in $\Gamma_u(T)$ when in state $i$.
Equivalently, the hyperarcs are
\begin{itemize}
  \item in $\Hcal_u^+(T)$, the pairs $(J,\{i\})$ such that $i \notin J$ and
    \[
      \forall \varepsilon > 0 , \quad T_i(u + \varepsilon \unit_J) > T_i(u) ;
    \]
  \item in $\Hcal_u^-(T)$, the pairs $(J,\{i\})$ such that $i \notin J$ and
    \[
      \forall \varepsilon > 0 , \quad T_i(u - \varepsilon \unit_J) < T_i(u) .
    \]
\end{itemize}
Similarly to \Cref{sec:hypergraphs-dominions-slice-spaces}, a set of nodes $J \subsetneq [n]$ is
invariant in $\Hcal_u^+(T)$ (resp., $\Hcal_u^-(T)$) if and only if $\cpt{J}$ is a dominion of
player \MIN (resp., player \MAX) in $\Gamma_u(T)$.
This allows us to reformulate \Cref{thm:uniqueness-additive} in terms of hypergraph reachability.

\begin{theorem}
  \label{thm:hypergraphs-uniqueness}
  Let $T: \R^n \to \R^n$ be a monotone additively homogeneous map.
  Then an eigenvector $u$ of $T$ is not unique, up to an additive constant,
  if and only if there exists a pair $(I,J)$ of nonempty disjoint subsets of $[n]$ such that
  $\reach(\cpt{I},\Hcal_u^+(T)) = \cpt{I}$ and $\reach(\cpt{J},\Hcal_u^-(T)) = \cpt{J}$.
  \qed
\end{theorem}

\subsubsection{The special case of convex maps}

When the map $T$ is convex, the latter characterization simplifies, along the same lines
as \Cref{sec:convexity-graph-slice-spaces}.
To that purpose, we introduce the ``local'' directed graph $\Gcal_u(T)$, with set of vertices
$[n]$ and an edge from $i$ to $j$ if
\[
  \forall \varepsilon > 0 , \quad T_i(u + \varepsilon \unit_{\{j\}}) > T_i(u) .
\]
Then the relation between the dominions of player \MIN in $\Gamma_u(T)$ and the final classes of
$\Gcal_u(T)$ is the same as with the game $\Gamma_\infty(T)$ and the digraph $\Gcal_\infty(T)$.
Specifically, any final class of $\Gcal_u(T)$ is a dominion of player \MIN in $\Gamma_u(T)$ and
conversely any dominion of \MIN in the latter game contains a final class of $\Gcal_u(T)$. 
This is a direct consequence of the definitions (recall that $\Delta$ is a dominion of
player \MIN in $\Gamma_u(T)$ if and only if there is some $\varepsilon > 0$ such that
$T_i(u + \varepsilon \unit_\cpt{\Delta}) = T_i(u)$ for all $i \in \Delta$) and the following
inequalities, which hold for every $\varepsilon > 0$, every $i \in [n]$, every subset
$J \neq [n]$ and every $j \notin J$:
\[
  T_i (u) \leq T_i(u + \varepsilon \unit_{\{j\}}) \leq T_i(u + \varepsilon \unit_\cpt{J}) \leq
  \frac{1}{|\cpt{J}|} \sum_{\ell \notin J} T_i(u + |\cpt{J}| \varepsilon \, \unit_{\{\ell\}}) .
\]
We mention that the first two inequalities come from the monotonicity of $T$ whereas
the last one stems from its convexity.
Consequently, \Cref{cor:hypergraphs-convex-slice-spaces} can be transposed to
the problem of uniqueness of the eigenvector.

\begin{corollary}
  Let $T: \R^n \to \R^n$ be a convex monotone additively homogeneous map.
  Then an eigenvector $u$ of $T$ is unique, up to an additive constant,
  if and only if the digraph $\Gcal_u(T)$ has a unique final class $C$ and
  $\reach(C,\Hcal_u^-(T)) = [n]$.
\end{corollary}

\subsubsection{Complexity aspects}

Up to the definition of the oracle, the complexity results stated in
\Cref{sec:complexity-slice-spaces} also readily adapt to the problem of uniqueness
of the eigenvector.
We next briefly state them.

Given a monotone additively homogeneous map $T: \R^n \to \R^n$ and a point $u \in \R^n$,
we introduce the oracle $\Omega_u(T)$ which takes as input a tuple $(J,i,\pm)$ and
returns a yes/no answer, the answer being ``yes'' if and only if
$\pm T_i(u \pm \varepsilon \unit_J) > \pm T_i(u)$ for all $\varepsilon > 0$.
Such an oracle allows us to check if a pair $(J,\{i\})$ is a hyperarc of $\Hcal_u^\pm(T)$,
or if $(i,j)$ is an edge of $\Gcal_u(T)$.
Then we have the following, which is a straightforward adaptation of
\Cref{thm:boundedness-decision} and \Cref{cor:boundedness-decision-convex},
respectively.

\begin{theorem}
  Assume that $T: \R^n \to \R^n$ is monotone and additively homogeneous and let $u$ be
  an eigenvector.
  \begin{enumerate}
    \item The problem of deciding whether $u$ is the unique eigenvector of $T$, up to an additive
      constant, can be solved in $O(2^n n^2)$ steps by a Turing machine with oracle $\Omega_u(T)$.
    \item If $T$ is convex, the latter problem can be solved in $O(n^2)$ steps by a Turing
      machine with oracle $\Omega_u(T)$.
  \end{enumerate}
\end{theorem}

\section{Applications}
\label{applications}

\subsection{Stochastic games}
\label{sec:stochgames}

A two-player zero-sum stochastic game $\Gamma$ is described by a state space, which we assume
here to be finite, say $[n]$; by action spaces, $A$ for Player I and $B$ for Player II;
by a payoff function $r: [n] \times A \times B \to \R$; and by a transition function $\rho$
from $[n] \times A \times B$ to the set of probabilities over $[n]$.
At each stage $\ell$, given the current state $i$, Player I (resp., II) chooses an action
$a$ in $A$ (resp., $b$ in $B$).
This incurs a stage payoff $r_\ell = r(i,a,b)$ paid by Player I to Player II, and the next state
is drawn according to the distribution $\rho(\cdot \mid i,a,b)$.
In the game with imperfect information, the two players play simultaneously,
whereas in the perfect information game, one assumes that Player II selects
a current action after having observed the previous action of Player I,
and similarly for Player I. 

Given an initial state $i$ and a finite number $k$ of stages, Player I aims at minimizing
the Ces\`aro mean $\frac{1}{k} \sum_{\ell=0}^{k-1} g_\ell$, whereas Player II wants to
maximize it.
Under standard assumptions, in the imperfect information case, the $k$-stage game, played with randomized strategies, has a {\em value}, denoted by $v^k_i$,
equal to the unique payoff achieved (resp., approached) by Nash equilibria
(resp., $\varepsilon$-Nash equilibria). In the perfect information case, 
the value does exist even if we force the players to use deterministic strategies. A standard problem is to understand the asymptotic behavior of the value vector $v^k$.
We refer the reader to \cite{NS03} for background on stochastic games.

Using a dynamic programming principle, the value vector $v^k$ can be computed recursively
by means of the so-called ``Shapley operator'' of $\Gamma$.
The latter is a monotone additively homogeneous map $T: \R^n \to \R^n$ whose $i$th coordinate
is given by
\[
  T_i(x) = \inf_{a \in A} \sup_{b \in B} \; \bigg\{ r(i,a,b) +
    \sum_{j=1}^n x_j \, \rho(j \mid i,a,b) \bigg\} \enspace .
\]
In the imperfect information case, one generally assumes
that the $\inf$ and $\sup$ operators commute (which is guaranteed
by standard convexity/compactness assumptions).
In contrast, when dealing with a game with perfect information, 
the $\inf$ and $\sup $ operators need not commute. 
In both settings, the value vector is determined by the recursive formula
\[
  v^0 = 0 \quad \text{and} \quad (k+1) v^{k+1} = T(k v^k) .
\]

It is straightforward to check that if the ergodic equation~\labelcref{eq:ergodic-equation}
is solvable, then the {\em mean payoff} vector, given by
$\lim_{k \to \infty} v^k = \lim_{k \to \infty} \frac{1}{k} T^k(0)$, exists
and is equal to the constant vector $\lambda \unit$. 

\medskip
In \cite{AGH15}, the solvability of the ergodic equation~\labelcref{eq:ergodic-equation} has been
studied for Shapley operators of stochastic games with a bounded payoff function.
The results of \Cref{sec:existence,sec:existence-hypergraphs} extend the ones
in \cite{AGH15} (see in particular Theorems~3.1 and~5.3).
The following example illustrates the suboptimality of the latter results by exhibiting
a stochastic game (with unbounded payoffs) for which the Shapley operator has all
its slice spaces bounded in Hilbert's seminorm whereas its recession operator
(see the definition below) has nontrivial fixed points.

However, it is worth mentioning that if we assume that the action spaces of a stochastic game
$\Gamma$ with Shapley operator $T$ are compact and that the payoff and transition functions are
continuous (hence bounded), then the dominion condition in \Cref{thm:existence-additive},
which applies to the ``abstract'' game $\Gamma_\infty(T)$, is equivalent to the dominion condition
in the initial game $\Gamma$ (see \cite[Prop.~5.1]{AGH15}).
Likewise, it is possible to give a game-theoretical interpretation of the hypergraphs
$\Hcal_\infty^\pm(T)$:
\begin{itemize}
  \item a pair $(J,\{i\})$ is a hyperarc of $\Hcal_\infty^-(T)$ (resp., $\Hcal_\infty^+(T)$)
    if and only if, in the game $\Gamma$, Player I (resp., Player II) can force the state
    to move from $i$ to $J$ with positive probability;
  \item $\reach(J,\Hcal_\infty^-(T))$ (resp., $\reach(J,\Hcal_\infty^+(T))$) represents
    all the states from which $J$ can be made accessible by Player I (resp., Player II)
    in finite time, with positive probability.
\end{itemize}

\begin{example}
  We consider a stochastic game with unbounded payments inspired by the classical
  Blackmailer's Dilemma (see \cite{Whi83}).
  In the latter, the amount asked by a blackmailer to a victim influences the probability that
  the victim becomes resistant.
  The dynamic programming operator of the game is the monotone additively homogeneous map
  $T: \R^3 \to \R^3$ given by
  \begin{equation}
    \label{eq:blackmailer}
    T(x) = \left(
    \begin{gathered}
      \sup_{0 < p \leq 1} \big\{ \log p + p ( x_2 \wedge x_3 ) + (1-p) x_1 \big\} \\
      \inf_{0 < p \leq 1} \big\{ -\log p + p x_3 + (1-p) x_1 \big\} \\
      x_3
    \end{gathered}
    \right) ,
  \end{equation}
  where $\wedge$ stands for $\min$.
  This game has three states:
  the first player (Player I) partially controls state $1$, the second player
  (Player II) controls state $2$, and state $3$ is an absorbing state (i.e., a state in which
  the dynamics is stationary, whatever actions the players choose).
  More precisely, in state $1$, Player I chooses an action $p \in (0,1]$ and
  receives $\log p$ from Player II.
  Then, with probability $1-p$, the next state remains $1$, and with probability $p$,
  it is chosen by Player II between state $2$ and state $3$.
  Thus, maximizing the one-day payoff would lead to select $p=1$, but this leads to leave
  state $1$ with probability one.
  A dual interpretation applies to Player II in state $2$.
  
  Let us find out whether the optimality equation~\labelcref{eq:ergodic-equation} has a solution,
  and for that purpose, let us construct the ``abstract'' stochastic game $\Gamma_\infty(T)$
  defined in \Cref{sec:existence}.
  In order to determine the action spaces of $\Gamma_\infty(T)$, it is convenient to notice that
  $T_1(x) = h((x_2 \wedge x_3) - x_1) + x_1$ and that $T_2(x) = -h(x_1-x_3)+x_1$, where $h$ is
  the real-valued function defined by $h(z) = \sup_{0 < p \leq 1} \{ \log p + p z \}$.
  Further note that $h$ satisfies $h(z) = -1-\log(-z)$ for $z \leq -1$, and $h(z)=z$
  for $z \geq -1$.
  Then we get that the sets of actions of player \MIN in $\Gamma_\infty(T)$ are
  \begin{itemize}
    \item in state 1: $\{1,2\}$, $\{1,3\}$, $\{1,2,3\}$;
    \item in state 2: $\{3\}$, $\{1,3\}$, $\{2,3\}$, $\{1,2,3\}$;
    \item in state 3: $\{3\}$, $\{1,3\}$, $\{2,3\}$, $\{1,2,3\}$.
  \end{itemize}
  As for player \MAX, his action sets are
  \begin{itemize}
    \item in state 1: $\{2,3\}$, $\{1,2,3\}$;
    \item in state 2: $\{1,3\}$, $\{1,2,3\}$;
    \item in state 3: $\{3\}$, $\{1,3\}$, $\{2,3\}$, $\{1,2,3\}$.
  \end{itemize}
  Hence, the dominions of \MIN in $\Gamma_\infty(T)$ are $\{3\}$, $\{1,3\}$, $\{2,3\}$
  and $\{1,2,3\}$, whereas the dominions of \MAX are $\{3\}$ and $\{1,2,3\}$.
  It follows that the dominion condition is not satisfied since every two dominions of \MIN
  and \MAX, respectively, have a nonempty intersection.
  So, according to \Cref{thm:existence-additive}, all the slice spaces of
  $T$ are bounded in Hilbert's seminorm.
  As a consequence, the ergodic equation~\labelcref{eq:ergodic-equation} is solvable
  for all operators $g+T$ with $g \in \R^3$.
  
  \medskip
  Alternatively, one may construct the hypergraphs $\Hcal_\infty^\pm(T)$ associated with $T$.
  A concise representation of these hypergraphs is provided in \Cref{fig:Hpm}.
  Only the (hyper)arcs with minimal tail (with respect to the inclusion partial order)
  are represented.
  For instance, there is no arc from $\{2\}$ to $\{1\}$ in $\Hcal_\infty^+(T)$ since
  $T_1(\alpha \unit_2) = 0$ for all $\alpha \geq 0$.
  However, there is a hyperarc from $\{2,3\}$ to $\{1\}$, since
  $T_1(\alpha \unit_{\{2,3\}}) = \alpha$ for all $\alpha \geq 0$,
  which yields $\lim_{\alpha \to +\infty} T_1(\alpha \unit_{\{2,3\}}) = +\infty$.
  
  The nontrivial invariant subsets of $\Hcal_\infty^+(T)$ are $\{1\}$, $\{2\}$ and $\{1,2\}$,
  whereas for $\Hcal_\infty^-(T)$, the only nontrivial invariant subset is $\{1,2\}$.
  Hence, for every pair of nontrivial invariant subsets in $\Hcal_\infty^+(T)$ and
  $\Hcal_\infty^-(T)$, respectively, the intersection of their complements is nonempty.
  The conclusion then follows from \Cref{thm:hypergraphs-slice-spaces}.
  
  \begin{figure}[htbp]
    \begin{center}
      \begin{minipage}[b]{0.33\linewidth}
        \centering
        \begin{tikzpicture}
          [scale=1,on grid,auto,bend angle=40,
            state/.style={circle,draw,inner sep=0pt,minimum size=0.4cm},
            head/.style={<-,>=angle 60},
            tail/.style={->,>=angle 60},
          hyperedge/.style={>=angle 60}]
          \node[state] (s1) at (0,0) {$1$};
          \node[state] (s2) at (-0.75,1.3) {$2$};
          \node[state] (s3) at (-1.5,0) {$3$}
          edge [tail,bend right,black!0] (s1)
          edge [tail,bend left] (s2);
          \hyperedgewithangles[0.35][$(hyper@tail)!0.65!(hyper@head)$]{s2/-90,s3/30}{-30}{s1/150};
          \node[state,color=black!0] (s5) at (-2.25,+1.5) [label=below:{$\Hcal_\infty^+(T)$}] {};
        \end{tikzpicture}
      \end{minipage}
      \hspace{3em}
      \begin{minipage}[b]{0.33\linewidth}
        \centering
        \begin{tikzpicture}
          [scale=1,on grid,bend angle=40,auto,
            state/.style={circle,draw,inner sep=0pt,minimum size=0.4cm},
            head/.style={<-,>=angle 60},
          tail/.style={->,>=angle 60}]
          \node[state] (s1) at (0,0) {$1$};
          \node[state] (s2) at (-0.75,1.3) {$2$}
          edge [tail,bend left] (s1)
          edge [head, bend right] (s1);
          \node[state] (s3) at (-1.5,0) {$3$}
          edge [tail,bend right] (s1)
          edge [tail,bend left] (s2);
          \node[state,color=black!0] (s5) at (-2.25,+1.5) [label=below:{$\Hcal_\infty^-(T)$}] {};
        \end{tikzpicture}
      \end{minipage}
    \end{center}
    \caption{The hypergraphs $\Hcal_\infty^\pm(T)$ associated with $T$~\labelcref{eq:blackmailer}.}
    \label{fig:Hpm}
  \end{figure}
  
  We finally mention that the same conclusion cannot be obtained from
  the results in \cite{GG04,CCHH10,AGH15}.
  Indeed, in these references, the solvability of the ergodic equation, or alternatively
  the boundedness of all the slice spaces, holds if the {\em recession operator} of $T$,
  the self-map of $\R^n$ defined by $\hat{T}(x) = \lim_{k \to \infty} k^{-1} \, T(k x)$,
  has only constant fixed points (i.e., proportional to the unit vector $\unit$).
  Here, the recession operator is given by
  \[
    \hat{T}(x) = 
    \begin{pmatrix}
      x_1 \vee (x_2 \wedge x_3) \\ x_1 \wedge x_3 \\ x_3
    \end{pmatrix} .
  \]
  Since any vector $\transpose{(\alpha,0,0)}$ with $\alpha \geq 0$ is a fixed point of $\hat{T}$,
  then the sufficient condition appearing in the latter references is not satisfied.
\end{example}

\subsection{Generalized means}
\label{sec:class-M}

We next apply our results to the class of generalized means considered in
\cite{Nus88,LN12,LLN18}.
For any scalar $r \in \R \setminus \{0\}$ and any stochastic vector $\sigma \in \R^n$
(i.e., $\sigma_i \geq 0$ for all indices $i \in [n]$ and $\sum_i \sigma_i = 1$),
let $M_{r \sigma}(x)$ be the $(r,\sigma)$-mean of any vector $x \in \Rplus^n$, defined by
\[
  M_{r \sigma}(x) := \Bigg( \sum_{i \in [n]} \sigma_i x_i^r \Bigg)^{1/r} .
\]
We let $\supp(\sigma) := \{ i \in [n] \mid \sigma_i > 0 \}$ be the support of $\sigma$ and
define, by continuity,
\begin{align*}
  M_{0 \sigma}(x) & := \prod_{i \in \supp(\sigma)} x_i^{\sigma_i} , \\
  M_{+\infty \sigma}(x) & := \max_{i \in \supp(\sigma)} x_i , \\
  M_{-\infty \sigma}(x) & := \min_{i \in \supp(\sigma)} x_i .
\end{align*}
When $\sigma$ is the uniform probability vector in $\Rplus^n$ (i.e., with entries equal
to $1/n$), we write $M_{r}(x) := M_{r \sigma}(x)$ for brevity.

We define the set $\Mbar_{n1}$ consisting of maps $\R_+^n\to \R$ given by well-formed
expressions involving the mean operations, the multiplication by a nonnegative scalar,
and the variables $x_1, \dots, x_n$.
We define $\Mbar_{n n}$ to be the set of maps $\R_+^n\to \R_+^n$ whose coordinates
belong to $\Mbar_{n1}$. 
For instance, the map
\begin{align}
  h (x_1,x_2) & =
  M_{+\infty} \Big( M_{-3} \big( M_{-\infty}(x_1,2x_2), \pi x_1 \big),
  18 M_{0,(1/4,3/4)}(x_1,x_2) \Big)
  \label{e-expr} \\
  & = \max \Big( \big( \min(x_1,2x_2)^{-3}+(\pi x_1)^{-3} \big)^{-1/3},
  18 \sqrt[4]{x_1x_2^3} \Big)
  \nonumber
\end{align}
belongs to $\Mbar_{21}$ and the map $f(x_1,x_2) = ( h(x_1, x_2), x_2 )$ belongs to $\Mbar_{22}$.

The {\em signature} of a well-formed expression defining a map $f$ in $\Mbar_{n1}$ is
the map of $\Mbar_{n1}$ obtained by applying the following operations to this expression:
\begin{itemize}
  \item delete the multiplicative constants;
  \item replace every occurrence of $M_{r \sigma}$ with $r>0$ (resp., $r<0$) by
    $M_{+\infty \sigma}$ (resp., $M_{-\infty \sigma}$);
  \item replace every occurrence of $M_{0 \sigma}$ by the uniform geometric mean of
    the arguments appearing in the support of $\sigma$.
\end{itemize}
For instance, the signature of the expression in~\labelcref{e-expr} is the map
\begin{equation*}
  (x_1, x_2) \mapsto \max \Big( \min \big( \min(x_1, x_2), x_1 \big) , \sqrt{x_1x_2} \Big) =
  \sqrt{ x_1 x_2} .
\end{equation*}
The signature of a vector-valued expression is defined entrywise.

\begin{theorem}
  The validity of the dominion condition (see \Cref{it:dominion-condition} in
  \Cref{thm:existence-multiplicative}) for a map $f\in \Mbar_{nn}$ given by a well-formed
  expression depends only of the signature of this expression.
  Moreover, the dominion condition holds for $f$ if and only if all the maps
  $g \in \Mbar_{nn}$ that share a common signature with $f$ have a positive eigenvector. 
\end{theorem}

\begin{proof}
  The limits $\lim_{\alpha \to \pm \infty} f_i(\exp(\alpha \unit_{J^c}))$ are easily seen
  to depend only on the signature of $f$.
  Hence, if the dominion condition holds for $f$, then it holds for all the maps $g$
  that share a common signature with $f$.
  By \Cref{cor:solvability-eigenproblem}, all these maps have a positive eigenvector.

  Conversely, the latter property implies in particular that all the maps of the form
  $g = Df$, where $D$ is a diagonal matrix with positive diagonal entries, have
  a positive eigenvector.
  By \Cref{thm:existence-multiplicative}, $f$ satisfies the dominion condition.
\end{proof}

Hence, for maps in $\Mbar_{nn}$, the boundedness of {\em all} slice spaces (or
the existence of a positive eigenvector independently of the numerical values of
the parameters of the map) is algorithmically decidable.
Checking the existence of {\em one} nonempty and bounded slice space for a given map, or
equivalently, checking whether the eigenspace is nonempty and bounded, is a much harder problem.
A semidecision procedure is provided in \cite{LLN18}.
The decidability of this problem would follow from the conjectured decidability of
the real exponential field \cite{wilkie}.

\subsection{Nonnegative tensors}

Consider a $d$-order $n$-dimensional tensor $\Tsr$ defined by $n^d$ real entries,
$a_{i_1 \dots i_d}$ for $i_1,\dots,i_d \in [n]$.
It yields a self-map $f$ of $\R^n$, whose $i$th coordinate function is given by
\[
  f_i(x) = \big[ \Tsr x^{(d-1)} \big]_i :=
  \sum_{1 \leq i_2, \dots, i_d \leq n} a_{i \, i_2 \dots i_d} \; x_{i_2} \dots x_{i_d} .
\]
The tensor eigenvalue problem introduced by Lim \cite{Lim05} and Qi \cite{Qi05}
asks for the existence of an eigenvalue $\lambda \in \R$ and an eigenvector
$u \in \R^n$ solution of
\begin{equation}
  f(u) = \Tsr u^{(d-1)} = \lambda u^{d-1} ,
  \label{eq:tensor-eigenproblem}
\end{equation}
where $u^{d-1} := (u_1^{d-1}, \dots, u_n^{d-1})$.
If the tensor $\Tsr$ is nonnegative, meaning that $a_{i_1 \dots i_d} \geq 0$
for all multi-indices, a variant of this problem is the existence of a positive eigenvalue
$\lambda > 0$ and a positive eigenvector $u \in \Rplus^n$.

Lim showed that a nonnegative tensor $\Tsr$ has a positive eigenvalue and a unique positive
eigenvector (up to a scaling) if $\Tsr$ is {\em irreducible}, meaning that $f$ does not leave
invariant a nontrivial face of the positive orthant (see \cite[Thm.~1]{Lim05}).
Friedland, Gaubert and Han \cite{FGH13} showed that the same conclusion holds under
a milder condition, {\em weak irreducibility}, arising from \cite{GG04}.
The condition in the latter reference requires the strong connectivity of the directed graph
$\Gcal_\infty(\Tsr)$ defined by the set of nodes $[n]$ and an edge from $i$ to $j$ if
$\lim_{\alpha \to +\infty} f_i \big( \exp(\alpha \unit_{\{j\}}) \big) = +\infty$.
Alternatively, there is an edge from $i$ to $j$ if and only if there exists a set of indices
$(i_2, \dots, i_d)$ containing $j$ and such that $a_{i \, i_2 \dots i_d} > 0$.

\Cref{cor:hypergraphs-convex-slice-spaces} leads to a refinement of these results.
Let us introduce the hypergraph $\Hcal_\infty(\Tsr)$ corresponding to the hypergraph
$\Hcal_\infty^-(\log \circ f \circ \exp)$, as defined in \Cref{sec:algorithmic-aspects}.
Precisely, the set of nodes is $[n]$ and there is a hyperarc from a subset $J \subset [n]$ to
$\{i\}$ if $i \notin J$ and $\lim_{\alpha \to -\infty} f_i \big( \exp(\alpha \unit_J) \big) = 0$.
Equivalently, the pair $(J,\{i\})$ is a hyperarc of $\Hcal_\infty(\Tsr)$ if $i \notin J$ and
we have
\begin{equation}
  \label{eq:hyperarcs-tensors}
  \forall (i_2, \dots, i_d) \in [n]^{d-1}, \quad
  \big( \enspace a_{i \, i_2 \dots i_d} > 0 \implies J \cap \{i_2, \dots, i_d\} \neq \emptyset
  \enspace \big) .
\end{equation}
We define the \emph{pattern} of the tensor $\mathcal{A}$ to be the set of multi-indices
$(i_1,\dots,i_d)\in [n]$ such that $a_{i_1\dots i_d}>0$. 

\begin{corollary}
  \label{cor:eigenproblem-tensor}
  Let $\Tsr$ be a nonnegative $n$-dimensional tensor.
  Every nonnegative tensor with the same pattern as $\Tsr$ has
  a positive eigenvector if and only if the directed graph $\Gcal_\infty(\Tsr)$ has
  a unique final class $C$ and $\reach(C,\Hcal_\infty(\Tsr)) = [n]$.
\end{corollary}

\begin{proof}
  Consider $T = (d-1)^{-1} \, \log \circ f \circ \exp$.
  This is a monotone additively homogeneous self-map of $\R^n$.
  Furthermore, any eigenpair $(\mu,v) \in \R \times \R^n$ of $T$ yields an eigenpair
  $(\lambda,u)$ of $f$ with the required properties, namely
  $\lambda = \operatorname{e}^{\mu (d-1)} > 0$ and $u = \exp(v) \in \Rplus^n$.
  It is a standard result that functions such as $\log \circ f_i \circ \exp$ are convex
  (see e.g.~\cite[Ex.~2.16, Ex.~2.52]{RW98}).
  Hence $T$ is convex.
  Now, by definition, the directed graph $\Gcal(T)$ and the hypergraph $\Hcal_\infty^-(T)$
  (see \Cref{sec:algorithmic-aspects}) are the same as $\Gcal_\infty(\Tsr)$ and
  $\Hcal_\infty(\Tsr)$, respectively.
  Furthermore, by construction, the latter graphs only depend on the pattern of $\Tsr$.
  The conclusion then follows from \Cref{cor:hypergraphs-convex-slice-spaces} and
  the ``stability'' result \Cref{thm:existence-additive}.
\end{proof}

\begin{example}
  Consider the nonnegative tensor $\Tsr$ of dimension 4 and order 3 whose positive entries are:
  \begin{equation*}
     a_{1 1 2} , \enspace a_{1 2 2} , \enspace
     a_{2 1 1} , \enspace a_{2 1 2} , \enspace a_{2 2 2} , \enspace
     a_{3 1 1} , \enspace a_{3 1 2} , \enspace a_{3 2 3} , \enspace
     a_{4 1 4} , \enspace a_{4 3 3} .
  \end{equation*}
%   \begin{equation*}
%     \begin{split}
%       & a_{1 1 2} , \enspace a_{1 2 2} , \\
%       & a_{2 1 1} , \enspace a_{2 1 2} , \enspace a_{2 2 2} , \\
%       & a_{3 1 1} , \enspace a_{3 1 2} , \enspace a_{3 2 3} , \\
%       & a_{4 1 4} , \enspace a_{4 3 3} .
%     \end{split}
%   \end{equation*}
  An instance of this tensor, with all nonzero coefficients equal to $1$, is represented
  by the following self-map of $\R^4$:
  \[
    f(x) = 
    \begin{pmatrix}
      x_1 x_2 + x_2^2 \\
      x_1^2 + x_1 x_2 + x_2^2 \\
      x_1^2 + x_1 x_2 + x_2 x_3 \\
      x_1 x_4 + x_3^2
    \end{pmatrix} .
  \]
  To check whether $\Tsr$ has a positive eigenvector for any numerical values of its parameters
  (provided they are positive), let us construct the graph $\Gcal_\infty(\Tsr)$ and
  the hypergraph $\Hcal_\infty(\Tsr)$.

  Their set of nodes is $\{1,2,3,4\}$.
  In $\Gcal_\infty(\Tsr)$, the edges going out of node 1, for instance, are $(1,1)$ and $(1,2)$
  since $a_{1 1 2}$ and $a_{1 2 2}$ are the only positive entries of $\Tsr$
  of the form $a_{1 i j}$.
  \Cref{fig:graphs-tensor}, on the left, show a representation of $\Gcal_\infty(\Tsr)$
  without loops.
  In $\Hcal_\infty(\Tsr)$, there is no hyperarc with head $\{2\}$, for instance, since
  the subset $\{1,2\}$ (which contains 2) is the smallest one which satisfies
  condition~\labelcref{eq:hyperarcs-tensors}.  
  However, there is a hyperarc from $\{1,2\}$ to $\{3\}$ since for every positive entry
  $a_{3 i j}$, one of the indices $i, j$ is either 1 or 2.
  \Cref{fig:graphs-tensor}, on the right, shows a concise representation of
  $\Hcal_\infty(\Tsr)$ where only the (hyper)arcs with minimal tail (with respect to the inclusion
  partial order) are represented.

  \begin{figure}[htbp]
    \begin{center}
      \begin{minipage}[b]{0.33\linewidth}
        \centering
        \begin{tikzpicture}
          [scale=1,on grid,bend angle=40,auto,
            state/.style={circle,draw,inner sep=0pt,minimum size=0.4cm},
            head/.style={<-,>=angle 60},
          tail/.style={->,>=angle 60}]
          \node[state] (s1) at (0,0) {$1$};
          \node[state] (s2) at (0.75,1.3) {$2$}
          edge [tail,bend right] (s1)
          edge [head,bend left] (s1);
          \node[state] (s3) at (-0.75,1.3) {$3$}
          edge [tail,bend right] (s1)
          edge [tail, bend left] (s2);
          \node[state] (s4) at (-1.5,0) {$4$}
          edge [tail,bend right] (s1)
          edge [tail,bend left] (s3);
          \node[state,color=black!0] (s5) at (-2.25,+1.5) [label=below:{$\Gcal_\infty(\Tsr)$}] {};
        \end{tikzpicture}
      \end{minipage}
      \hspace{3em}
      \begin{minipage}[b]{0.33\linewidth}
        \centering
        \begin{tikzpicture}
          [scale=1,on grid,auto,bend angle=40,
            state/.style={circle,draw,inner sep=0pt,minimum size=0.4cm},
            head/.style={<-,>=angle 60},
            tail/.style={->,>=angle 60},
          hyperedge/.style={>=angle 60}]
          \node[state] (s1) at (0,0) {$1$};
          \node[state] (s2) at (0.75,1.3) {$2$}
          edge [tail,bend left] (s1);
          \node[state] (s3) at (-0.75,1.3) {$3$};
          \node[state] (s4) at (-1.5,0) {$4$}
          edge [tail,bend right,color=black!0] (s1);
          \hyperedgewithangles[0.35][$(hyper@tail)!0.65!(hyper@head)$]{s3/-90,s1/150}{210}{s4/30};
          \hyperedgewithangles[0.35][$(hyper@tail)!0.65!(hyper@head)$]{s2/210,s1/90}{150}{s3/-30};
          \node[state,color=black!0] (s5) at (-2.25,+1.5) [label=below:{$\Hcal_\infty(\Tsr)$}] {};
        \end{tikzpicture}
      \end{minipage}
    \end{center}
    \caption{The graph $\Gcal_\infty(\Tsr)$ and the hypergraph $\Hcal_\infty(\Tsr)$ associated
      with the nonnegative tensor $\Tsr$.}
    \label{fig:graphs-tensor}
  \end{figure}

  The graph $\Gcal_\infty(\Tsr)$ has a unique final class $C = \{1,2\}$.
  Furthermore, the set of reachable nodes from $C$ in $\Hcal_\infty(\Tsr)$ is $\{1,2,3,4\}$.
  Hence, we deduce from \Cref{cor:eigenproblem-tensor} that the tensor $\Tsr$ has
  a positive eigenvector for any numerical instance.
\end{example}

\section{Concluding remarks}

We finally point out three open questions which emerge from the present work.
First, the notion of dominion is inherently combinatorial and finite-dimensional,
it would be valuable to generalize our existence and uniqueness
results for nonlinear eigenvectors to the infinite-dimensional setting.
Next, the present game theory approach is related to the geometry of the standard orthant.
It would be of great interest to find combinatorial or geometric conditions
for the existence and uniqueness of nonlinear maps defined on other finite-dimensional
cones, especially the cone of positive semidefinite matrices.
Finally, the existence condition in terms of dominions characterizes the situation
in which, for all diagonal matrices with positive diagonal entries,
the perturbed map $Df$ has a positive eigenvector.
Finer conditions may be hoped for if one relaxes the requirement
to find criteria invariant under this family of perturbations. 

\bibliographystyle{alpha}
\bibliography{references}

\end{document}